\documentclass[12pt,twoside]{amsart}
\usepackage{amsmath,amsthm,amssymb}
\usepackage{mathrsfs, array,tikz}
\usetikzlibrary{arrows}

\usepackage{caption,subcaption}
\usepackage{bbold}
\usepackage{url}
\usepackage{amsaddr}

\usetikzlibrary{matrix,positioning,decorations.pathreplacing}

\newtheorem{thm}{Theorem}[section]
\newtheorem{cor}[thm]{Corollary}

\newtheorem{lemma}[thm]{Lemma}

\newtheorem{prop}[thm]{Proposition}
\newtheorem{rmk}[thm]{Remark}

\theoremstyle{definition}
\newtheorem{defi}[thm]{Definition}

\DeclareMathOperator{\Gr}{Gr}

\DeclareMathOperator{\cl}{cl}

\begin{document}

\title{Ternary and quaternary positroids}
\author{Jeremy Quail}
\address{Dept.\:of Mathematics \& Statistics \\ University of Vermont \\ Burlington, VT, USA}
\email{jeremy.quail@uvm.edu}
\date{July 10, 2024}

\maketitle

\begin{abstract}A positroid is an ordered matroid realizable by a real matrix with all nonnegative maximal minors. Postnikov gave a map from ordered matroids to Grassmann necklaces, for which there is a unique positroid in each fiber of the map. Here, we give forbidden minor characterizations of ternary and quaternary positroids. We show that a positroid is ternary if and only if it is near-regular, and that all ternary positroids are formed by direct sums and $2$-sums of binary positroids and positroid ordered whirls. We prove that a positroid is quaternary if and only if it is $U^2_6, U^4_6,$ and $P_6$-free. Under the map from ordered matroids to Grassmann necklaces, we fully characterize the fibers of ternary positroids, referred to as their positroid envelope classes; in particular, the envelope class of a positroid ordered whirl of rank-$r$ contains exactly four matroids.
\end{abstract}

\section{Introduction}

A positroid on $n$ elements of rank-$r$ is an ordered matroid that can be realized by an $r \times n$ full-rank real matrix with all nonnegative maximal minors. They were introduced by Postnikov in the study of the nonnegative Grassmannian, and they have connections to many interesting combinatorial objects~\cite{postnikov2006total}, as well as applications in physics~\cite{arkani2012scattering,kodama2014kp}. The positroids partition the ordered matroids into positroid envelope classes, and the intersection of these classes with the $\mathbb{R}$-realizable ordered matroids decomposes the real Grassmannians into open positroid varieties. The goal of this paper is to characterize the $\mathbb{F}_3$- and $\mathbb{F}_4$-realizable positroids, and to explore the poset structure of the positroid envelope classes.

Section~\ref{sec:prelim} contains basic definitions and necessary background. We consider the three classes of combinatorial objects of interest and the maps between them: ordered matroids, decorated permutations, and Grassmann necklaces. We define the positroid envelope classes and leverage results on rank-preserving weak maps of matroids to give direct sum and $2$-sum decompositions of envelope classes. We consider the natural poset structure on positroid envelope classes given by rank-preserving weak maps, and show that if a positroid decomposes under direct sums and $2$-sums into matroids of rank or corank at most two then its envelope class is a join-semilattice. Finally, several equivalent characterizations of the binary positroids are given.

In Section~\ref{sec:main}, we give forbidden minor characterizations of the ternary and quaternary positroids, and we fully characterize the ternary positroid envelope classes. In particular, we show that non-binary $3$-connected ternary positroids must be isomorphic to whirls, and their envelope classes contain exactly four matroids. In general, ternary positroid envelope classes are lattices and ternary positroids can be decomposed via direct sums and $2$-sums into circuits, cocircuits and whirls.

\section{Preliminaries}\label{sec:prelim}

\subsection{Basic definitions}
A \emph{matroid} $M = (E,\mathcal{B})$ consists of a finite ground-set $E$ and a non-empty collection of subsets $\mathcal{B} \subseteq 2^E$, called bases, that satisfy the following \emph{basis exchange property}: 
\begin{itemize}
    \item if $B_1,B_2 \in \mathcal{B}$ and $x \in B_1 \setminus B_2$, then there exists $y \in B_2 \setminus B_1$ such that $(B_1 \setminus \{x\}) \cup \{y\} \in \mathcal{B}$.
\end{itemize}
The collection of \emph{independent sets} of $M$ is $\mathcal{I}(M) := \{ I \subseteq B: B \in \mathcal{B}(M) \}$, and the \emph{dependent sets} of $M$ are all subsets of $E(M)$ that are not independent. Minimal dependent sets of $M$ are called \emph{circuits}, and an element that constitutes a circuit whose cardinality is one is called a \emph{loop}. The \emph{dual} of a matroid $M$ is the matroid $M^*$ whose ground-set is $E(M^*) := E(M)$ and whose bases are $\mathcal{B}(M^*) := \{ E(M) \setminus B : B \in \mathcal{B}(M) \}$. A \emph{cocircuit} of $M$ is a circuit in $M^*$, and a \emph{coloop} of $M$ is a loop in $M^*$. The \emph{rank function} of a matroid $M$ is a map $r_M: 2^{E(M)} \to \mathbb{Z}_{\geq 0}$ given by $r_M(X) := \max \{ \lvert I \rvert : I \subseteq X, I \in \mathcal{I}(M) \}$. We call $X \subseteq E(M)$ a \emph{flat} of $M$ if for all $y \in E(M) \setminus X, r_M(X \cup \{y\}) > r_M(X)$. A flat $X$ of $M$ is called a \emph{hyperplane} if $r_M(X) = r_M(M) - 1$. We call $X \subset E(M)$ a \emph{circuit-hyperplane} if $X$ is both a circuit and a hyperplane of $M$. The \emph{closure operator} of a matroid $M$ is a function $\cl : 2^{E(M)} \to 2^{E(M)}$ given by, for all $X \subseteq E(M)$, $\cl(X) := \{e \in E(M) : r_M(X \cup \{e\}) = r_M(X) \}$. A \emph{matroid isomorphism} between two matroids $M$ and $N$ is a basis preserving bijection $E(M) \leftrightarrow E(N)$.

From a matroid $M$ we can obtain a new matroid $M \setminus e$ by \emph{deleting} the element $e \in E(M)$. The ground-set of $M \setminus e$ is given by $E(M \setminus e) := E(M) \setminus \{e\}$. The independent sets of $M \setminus e$ are given by $\mathcal{I}(M \setminus e) := \{I : e \notin I \in \mathcal{I}(M)\}$. We can obtain a new matroid $M/e$ by \emph{contracting} the element $e \in E(M)$. The ground-set of $M/e$ is given by $E(M/e) := E(M) \setminus \{e\}$. The bases of $M/e$ are given by
\[
    \mathcal{B}(M/e) :=
    \begin{cases}
        \mathcal{B}(M), & \text{if $e$ is a loop}\\
        \{B \setminus \{e\} : e \in B \in \mathcal{B}(M)\}, & \text{otherwise.}
    \end{cases}
\]
A matroid $N$ obtained from $M$ by a sequence of deletions and contractions is called a \emph{matroid minor of $M$}. We say that the matroid $M$ is \emph{$N$-free} if $M$ does not contain a minor isomorphic to the matroid $N$.

Let $A$ be a matrix with entries over a field $\mathbb{F}$. We can obtain a matroid from $A$, denoted by $M(A)$, by taking the ground-set to be the column vectors of $A$ and the bases to be the collection of maximal linearly independent subsets of the column vectors. We say that $A$ is a matrix that \emph{realizes} the matroid $M(A)$. If there is an isomorphism between a matroid $M$ and $M(A)$, then we also say that $A$ realizes $M$. A matroid is called \emph{$\mathbb{F}$-linear}, or \emph{$\mathbb{F}$-realizable}, if it can be realized by a matrix with entries over $\mathbb{F}$. If a matroid $M$ is $\mathbb{F}$-linear, then so to is $M^*$. A matroid is called \emph{binary}, \emph{ternary}, or \emph{quaternary} if it is $\mathbb{F}_2$-, $\mathbb{F}_3$-, or $\mathbb{F}_4$-linear, respectively. A matroid $M$ is \emph{regular} if it is $\mathbb{F}$-linear for all fields $\mathbb{F}$. For any field $\mathbb{F}$, the class of $\mathbb{F}$-linear matroids is minor-closed, thus can be characterized by a (possibly infinite) set of forbidden minors. For example, a matroid is binary if and only if it has no minor isomorphic to $U^2_4$~\cite[Theorem 6.5.4]{oxley2006matroid}. The binary positroids are characterized in Theorem~\ref{thm:pos-graphic}.

Let $G = (V,E)$ be a graph consisting of a vertex-set $V$ and an edge-set $E$. We allow for graphs to have loops and multiple edges connecting two vertices. From $G$ we can obtain a matroid, denoted by $M(G)$, called the \emph{cycle matroid of $G$} by taking the ground-set to be the edge-set of $G$ and the bases to be the maximal forests in $G$. We say that $G$ is a graph that \emph{represents} any matroid isomorphic to $M(G)$. We say that a matroid is \emph{graphic} if it can be represented by some graph. The graphic positroids are characterized in Theorem~\ref{thm:pos-graphic}.

\subsection{Positroids}
An \emph{ordered matroid} $M$ is a matroid whose ground-set $E(M)$ has a total ordering. Recall that the set of maximal minors of a matrix $A$ is the collection of determinants of all maximal square submatrices of $A$. If $A$ is a square matrix, then the unique maximal minor is the determinant of $A$.

\begin{defi}
    A \emph{positroid} is an $\mathbb{R}$-linear ordered matroid that can be realized by a real matrix $A$ with all nonnegative maximal minors.
\end{defi}

We say that a matroid $M$ has a \emph{positroid ordering} if there exists a total ordering on $E(M)$ such that the induced ordered matroid obtained from $M$ is a positroid. Positroids are closed under duality and taking minors.

\begin{prop}[Proposition 3.5 in~\cite{ardila2016positroids}] \label{prop:pos-dual-minor}
    Let $M$ be a positroid on $\{i_1 < \ldots < i_n \}$. Then $M^*$ is a positroid on $\{i_1 < \ldots < i_n\}$. Furthermore, for any subset $S$ of $\{i_1 < \ldots < i_n\}$, the deletion $M \setminus S$ and the contraction $M/S$ are both positroids on $E(M) \setminus S$. Here the total order on $E(M) \setminus S$ is the one inherited from $\{i_1 < \ldots < i_n\}$.
\end{prop}

Let $X$ and $Y$ be subsets of a totally ordered set $Z = \{i_1 < i_2 < \cdots < i_n\}$. For a fixed $i_k \in Z$, we define a new total order $\leq_{i_k}$ on $Z$ given by 
\begin{equation*}
    i_k < i_{k+1} < \cdots i_n < i_1 < i_2 < \cdots i_{k-1}.
\end{equation*}
We say that $X$ and $Y$ are \emph{crossing subsets of $Z$} if there exist $w,y \in X$ and $x,z \in Y$ such that $w <_w x <_w y <_w z$. If $X$ and $Y$ are not crossing subsets of $Z$, then we call them \emph{non-crossing subsets of $Z$}. The following result characterizes positroids as those ordered matroids for which every disjoint circuit-cocircuit pair is non-crossing.

\begin{thm}[Theorem 5.1,5.2 in~\cite{ardila2017positively}; Chapter 4, Theorem 1.1 in~\cite{perez1987quelques}] \label{thm:pos-crossing}
    An ordered matroid $M$ on the totally ordered set $E(M)$ is a positroid if and only if for any circuit $C$ and any cocircuit $C^*$ satisfying $C \cap C^* = \emptyset$, the sets $C$ and $C^*$ are non-crossing subsets of $E(M)$.
\end{thm}

It follows that positroids are closed under dihedral action.

\begin{prop} \cite{quail2024positroid} \label{prop:pos-dihedral}
    Let $P = (E,\mathcal{B})$ be a positroid where $E = \{1 < 2 < \dots < n\}$, and let $D_n$ be the dihedral group of order $2n$. Then, for any $\omega \in D_n$, $(E, \{\omega[B] : B \in \mathcal{B}\})$ is a positroid.
\end{prop}

\subsection{Direct sums, 2-sums and tree decompositions}
In this section we define direct sums and 2-sums of matroids, through which we can decompose any matroid into its 3-connected components. These play an important role in Section~\ref{sec:main}, where we characterize the ternary positroids. Ternary positroids may have non-trivial 3-connected components, while binary (regular) positroids do not. 

Let $M$ be a matroid and $X \subseteq E(M)$. The \emph{connectivity function of} $M$, $\lambda_M$, is defined as
\begin{equation*}
    \lambda_M(X) := r_M(X) + r_M(E\setminus X) - r_M(M).
\end{equation*}
A $k$\emph{-separation of} $M$ is a pair $(X,E\setminus X)$ for which $\min\{ \lvert X \rvert, \lvert E \setminus X \rvert \} \geq k$ and $\lambda_M(X) < k$. For $n \geq 2$, a matroid $M$ is $n$\emph{-connected} if, for all $k \in [n-1]$, $M$ has no $k$-separations. We call a matroid $M$ \emph{disconnected} if it is not $2$-connected.

Let $M$ and $N$ be matroids on disjoint sets. The \emph{direct sum} of $M$ and $N$, $M \oplus N$, is the matroid whose ground-set is $E(M \oplus N) := E(M) \cup E(N)$, and whose bases are $\mathcal{B}(M \oplus N) := \{ B \cup B' : B \in \mathcal{B}(M), B' \in \mathcal{B}(N) \}$.

A matroid is disconnected if and only if it is the direct sum of two non-empty matroids. Every matroid $M$ decomposes under direct sums into a unique collection of non-empty $2$-connected matroids which we call the \emph{$2$-connected components} of $M$. Every positroid $P$ has a unique direct sum decomposition into $2$-connected positroids that give a non-crossing partition on the ground-set $E(P)$.

\begin{thm}[Theorem 7.6 in \cite{ardila2016positroids}] \label{thm:pos-direct-sum}
    Let $P$ be a positroid on $X$ and let $S_1, S_2, \ldots, S_t$ be the ground-sets of the $2$-connected components of $P$. Then $\{S_1, \ldots, S_t\}$ is a non-crossing partition of $X$. Conversely, if $S_1, S_2, \ldots, S_t$ form a non-crossing partition of $X$ and $P_1, P_2, \ldots, P_t$ are connected positroids on $S_1, S_2, \ldots, S_t$, respectively, then $P_1 \oplus \cdots \oplus P_t$ is a positroid.
\end{thm}

Let $M$ and $N$ be matroids, such that 
    \begin{itemize}
        \item[(i)] $\lvert E(M) \rvert, \lvert E(N) \rvert \geq 2$,
        \item[(ii)] $E(M) \cap E(N) = \{ e \}$, and
        \item[(iii)] neither $(e, E(M) \setminus \{ e\})$ nor $(e, E(N) \setminus \{e\})$ is a $1$-separator of $M$ or $N$ respectively.
    \end{itemize}
Then the \emph{$2$-sum} of $M$ and $N$, denoted by $M \oplus_2 N$, is the matroid on $E(M \oplus_2 N) = (E(M) \cup E(N)) \setminus \{e\}$ whose circuits are 
    \begin{align*}
        \mathcal{C}(M \setminus e)& \cup \mathcal{C}(N \setminus e) \\
        &\cup \{(C_1 \cup C_2) \setminus \{e\} : C_1 \in \mathcal{C}(M), C_2 \in \mathcal{C}(N), e \in C_1 \cap C_2\}.
    \end{align*}
We call a $2$-sum, $M = M_1 \oplus_2 M_2$, \emph{non-trivial} if $|E(M_1)| > 2$ and $|E(M_2)| > 2$.

A $2$-connected matroid is $3$-connected if and only if it cannot be obtained by a non-trivial $2$-sum. We can use direct sums and $2$-sums to decompose matroids into collections of non-empty $3$-connected matroids which we call the \emph{$3$-connected components}. A $2$-sum decomposition of a $2$-connected matroid can be represented by a vertex- and edge-labelled tree.

\begin{defi}[\cite{oxley2006matroid}]
    A \emph{matroid-labelled tree} is a tree $T$ with vertex set $\{M_1, M_2, \ldots, M_k\}$ for some positive integer $k$ such that
    \begin{itemize}
        \item[(i)] each $M_i$ is a matroid;
        \item[(ii)] if $M_i$ and $M_j$ are joined by an edge $e$ of $T$, then $E(M_i) \cap E(M_j) = \{e\}$, where $e$ is not a $1$-seperator of $M_i$ or $M_j$; and
        \item[(iii)] if $M_i$ and $M_j$ are non-adjacent, then $E(M_i) \cap E(M_j) = \emptyset$.
    \end{itemize}
    A \emph{tree decomposition} of a $2$-connected matroid $M$ is a matroid-labelled tree $T$ such that if $V(T) = \{M_1, M_2, \ldots, M_k\}$ and $E(T) = \{e_1, e_2, \ldots, e_{k-1}\}$, then
    \begin{itemize}
        \item[(i)] $E(M) = (E(M_1) \cup E(M_2) \cup \cdots \cup E(M_k)) \setminus \{e_1,e_2, \ldots, e_{k-1}\}$;
        \item[(ii)] $\lvert E(M_i) \rvert \geq 3$ for all $i$ unless $\lvert E(M) \rvert < 3$, in which case $k=1$ and $M_1 = M$; and
        \item[(iii)] $M$ is the matroid that labels the single vertex of $T/\{e_1,\ldots,e_{k-1}\}$.
    \end{itemize}
\end{defi}

\begin{thm}[Theorem 8.3.10 in~\cite{oxley2006matroid}] \label{thm:canon-tree}
    Let $M$ be a $2$-connected matroid. Then $M$ has a tree decomposition $T$ in which every vertex label is $3$-connected, a circuit, or a cocircuit, and there are no two adjacent vertices that are both labelled by circuits or both labelled by cocircuits. Moreover, $T$ is unique to within relabelling of its edges.
\end{thm}

For a given $2$-connected matroid $M$, we call a tree decomposition given by Theorem~\ref{thm:canon-tree} a \emph{canonical tree decomposition} for $M$. The subsets of elements of a positroid across a direct sum must be non-crossing. This is not the case for $2$-sums. However, across a $2$-sum of a positroid the subsets of elements \emph{may} be non-crossing. More precisely, a $2$-connected positroid that is not $3$-connected can be written as a $2$-sum of positroids on non-crossing sets of elements. Extending this to canonical tree decompositions, we have the following theorem.

\begin{thm}[Theorem 2.27 in~\cite{quail2024positroid}] \label{thm:pos-canon-tree-iff}
    Let $M$ be a $2$-connected matroid on a subset of a densely totally ordered set $X$. Then $M$ is a positroid if and only if there exists a canonical tree decomposition $T$ with the following properties. Every vertex of $T$ is a positroid, and cutting $T$ along any edge $e$ results in two 2-connected positroids on non-crossing ground-sets.
\end{thm}

\subsection{Decorated permutations and Grassmann necklaces}

A \emph{decorated permutation} on a totally ordered set $X = \{i_1 < i_2 < \cdots < i_n\}$ is a pair $\pi^: = (\pi,col)$ consisting of a permutation $\pi : X \to X$ and a function \[col : \{\mbox{fixed points of }\pi\} \to \{-1,1\}.\]
We define the inverse decorated permutation $(\pi^:)^{-1} = (\pi^{-1},-col)$. For a fixed point $i_j$ of $\pi^:$, we denote
\begin{equation*}
    \pi^:(i_j) =
    \begin{cases}
        \overline{i_j}, & \text{if } col(i_j) = -1\\
        \underline{i_j}, & \text{if } col(i_j) = 1.
    \end{cases}
\end{equation*}
Let $i_j,i_k \in X$ and $Y \subseteq X$. We take $\min_{\leq_{i_j}}(Y)$ to be the minimum element of $Y$ with respect to the the total order $\leq_{i_j}$. We take the cyclic interval $[i_j,i_k]$ to be the set of all elements $i_t \in X$ such that $i_j \leq_{i_j} i_t  \leq_{i_j} i_k$.

The following map \{ordered matroids on $X$\} $\to$ \{decorated permutations on $X$\} is equivalent to the composition of two maps given by Postnikov in~\cite{postnikov2006total}. Let $M$ be a matroid on $X$ and $i_j \in X$, then the decorated permutation associated to $M$, denoted by $\pi^:_M$, is given by

\begin{equation*}
    \pi^:_M(i_j) :=
    \begin{cases}
        \overline{i_j}, & \text{if } i_k \text{ is a coloop}\\
        \underline{i_j}, & \text{if } i_k \text{ is a loop}\\
        \min_{\leq i_j} \left\{ i_k \in X : i_j \in \cl\big([i_{j+1},i_k]\big) \right\}, & \text{otherwise.}
    \end{cases}
\end{equation*}
This map interacts nicely with matroid duality, as shown in the following result.
\begin{cor}[Corollary 13 in~\cite{oh2009combinatorics}] \label{cor:dec-perm-dual}
For any ordered matroid $M$, we have that $\pi^: _M=(\pi^:_{M^*})^{-1}$.
\end{cor}

Let us now consider the interactions between decorated permutations and the matroid operations of direct sums and $2$-sums. Constructions similar to the ones we present have been given in~\cite[Section 4]{moerman2021grass} and~\cite[Section 12]{parisi2021m}. Ours differ by considering decorated permutations on any finite subset of a totally ordered set, not necessarily $[n]$ for some integer $n$. This is intended to deal with the technical challenges of decompositions of decorated permutations induced by direct sum and $2$-sum decompositions of their corresponding matroids. We begin by defining the following binary operation on decorated permutations.

\begin{defi}
    Let $X$ be a totally ordered set and let $\pi^:_1, \pi^:_2$ be decorated permutations on $X_1$ and $X_2$ respectively, such that $X_1$ and $X_2$ are disjoint subsets of $X$. We define the \emph{disjoint union} of $\pi^:_1$ and $\pi^:_2$, which we denote by $\pi^:_1 \sqcup \pi^:_2$, as the decorated permutation given by
    \begin{equation*}
        (\pi^:_1 \sqcup \pi^:_2)(i) =
        \begin{cases}
            \pi^:_1(i), & \text{if } i \in X_1\\
            \pi^:_2(i), & \text{if } i \in X_2.
        \end{cases}
    \end{equation*}
\end{defi}

\begin{rmk} \label{rmk:dec-perm-direct-sum}
    Let $M$ and $N$ be matroids on disjoint subsets of a totally ordered set $X$. Then,
    \begin{equation*}
        \pi^:_{M \oplus N} = \pi^:_M \sqcup \pi^:_N.
    \end{equation*}
\end{rmk}

We now turn our attention to the decorated permutations of the $2$-sums of matroids. Our construction is analogous to the amalgamations of decorated permutations in~\cite[Section 4]{moerman2021grass}

\begin{lemma}[\cite{parisi2021m,quail2024positroid}]\label{dec-perm-2sum}
    Let $M$ and $N$ be matroids on subsets of $X = \{i_1 < i_2 < \cdots < i_n\}$ such that $\lvert E(M) \rvert \geq 2$, $\lvert E(N) \rvert \geq 2$, and $E(M) \cap E(N) = \{i_e\}$. Suppose that $E(M)$ and $E(N)$ are non-crossing subsets of $X$. Suppose further that $i_e$ is a loop or coloop of $M$ or $N$, then $\pi^:_{M \oplus_2 N}$ is given by
    \begin{equation*}
        \pi^:_{M \oplus_2 N} = \pi^:_{M \setminus i_e \oplus N \setminus i_e}.
    \end{equation*}
    Suppose instead that $i_e$ is neither a loop nor a coloop of $M$ or $N$, then $\pi^:_{M \oplus_2 N}$ is given by
    \begin{equation*}
        \pi^:_{M \oplus_2 N}(i_j) :=
        \begin{cases}
            \pi^:_{M}(i_j), & \text{if } i_j \in E(M) \setminus E(N), \pi^:_{M}(i_j) \neq i_e\\
            \pi^:_{N}(i_e), & \text{if } i_j \in E(M) \setminus E(N), \pi^:_{M}(i_j) = i_e\\
            \pi^:_{N}(i_j), & \text{if } i_j \in E(N) \setminus E(M), \pi^:_{N}(i_j) \neq i_e\\
            \pi^:_{M}(i_e), & \text{if } i_j \in E(N) \setminus E(M), \pi^:_{N}(i_j) = i_e.
        \end{cases}
    \end{equation*}
\end{lemma}

A \emph{Grassmann necklace} on a totally ordered set $X = \{i_1 < i_2 < \cdots < i_n\}$ is a sequence $\mathcal{J} = (J_1,\ldots,J_n)$ of subsets $J_j \subseteq X$ given by
\[
    J_{j+1} =
    \begin{cases}
        (J_j \setminus \{i_j\}) \cup \{i_k\}, & \text{if } i_j \in J_j\\
        J_j, & \text{otherwise,}
    \end{cases}
\]
where the indices are taken modulo $n$. Postnikov defines in~\cite{postnikov2006total} the map \{ordered matroids on $X$\} $\to$ \{Grassmann necklaces on $X$\} as follows. Let $M$ be a matroid on $X$, then the Grassmann necklace associated to $M$ is $\mathcal{J}(M) = (J_1,\ldots,J_n)$, where $J_j$ is the lexicographically minimal basis with respect to $\leq_{i_j}$.

When the maps \{ordered matroids on $X$\} $\to$ \{decorated permutations on $X$\} and \{ordered matroids on $X$\} $\to$ \{Grassmann necklaces on $X$\} are restricted to the positroids on $X$ they give bijections, as shown in~\cite{postnikov2006total}. Furthermore, it follows from~\cite{postnikov2006total} that these maps commute. We illustrate this commutativity in Figure~\ref{fig:commute-triangle}.

\begin{figure}
    \begin{tikzpicture}
    \tikzset{edge/.style = {->,> = latex'}}
        \node (1) at (0,2){\{ordered matroids on $X$\}};
        \node (3) at (7,2){\{decorated permutations on $X$\}};
        \node (4) at (7,4){\{Grassmann necklaces on $X$\}};

        \draw[edge] (1) to (3);
        \draw[edge] (1) to (4);
        \draw[edge] (3) to (4);
        \draw[edge] (4) to (3);
    \end{tikzpicture}
    \caption{For any finite totally ordered set $X$, this diagram commutes.}
    \label{fig:commute-triangle}
\end{figure}

\subsection{Positroid envelope classes}

Let $M$ and $N$ be matroids. A \emph{weak map} $\varphi : M \to N$ is a bijection $E(M) \leftrightarrow E(N)$ such that for all $I \in \mathcal{I}(N)$, $\varphi^{-1}(I) \in \mathcal{I}(M)$. A \emph{rank-preserving weak map} is a weak map between matroids of the same rank. Rank-preserving weak maps give a partial order on matroids that share the same ground set. Suppose that $\mathbb{1} : M \to N$ is a rank-preserving weak map, then we take $M \geq N$. Equivalently, $M \geq N$ if $\mathcal{B}(N) \subseteq \mathcal{B}(M)$. Positroids can be characterized as the unique maximum, with respect to the rank-preserving weak map partial order, ordered matroid corresponding to a given Grassmann necklace. This is formalized in the following result.

\begin{thm}[Theorem 8 in~\cite{oh2011positroids}] \label{thm:pos-intersect}
    Let $P$ be a matroid on $X = \{i_1 < i_2 < \cdots < i_n\}$ of rank $k$ and $\mathcal{J}(P) = (J_1, J_2, \ldots, J_n)$ its corresponding Grassmann necklace. Then $P$ is a positroid if and only if
    \begin{equation*}
        \mathcal{B}(P) = \bigcap^n_{j=1} \left\{ B = {X \choose k} : J_j \leq_{i_j} B \right\}.
    \end{equation*}
\end{thm}

Let $M$ be an ordered matroid with corresponding Grassmann necklace $\mathcal{J}(M)$. The positroid $P$ is called the \emph{positroid envelope} of $M$ if $\mathcal{J}(P) = \mathcal{J}(M)$~\cite{knutson2013positroid}. The \emph{positroid envelope class of $P$}, denoted by $\Omega_P$, is the set of all ordered matroids whose positroid envelope is $P$. This implies that for all matroids $M \in \Omega_P$, $\mathbb{1} : P \to M$ is a rank-preserving weak map. It immediately follows from Corollary~\ref{cor:uni-free} that if a positroid $P$ is free of a uniform matroid $U^k_n$, then for all matroids $M \in \Omega_P$, $M$ is $U^k_n$-free.

\begin{cor}[Corollary 5.9 in~\cite{lucas1975weak}]\label{cor:uni-free}
Let $M$ and $N$ be matroids such that $\mathbb{1} : M \to N$ is a rank-preserving weak map. If $M$ is $U^k_n$-free then $N$ is $U^k_n$-free.
\end{cor}

Equivalently, for a positroid $P$, the envelope class of $P$ can be defined as $\Omega_P := \{M : \pi^:_M = \pi^:_P\}$. It immediately follows from Corollary~\ref{cor:dec-perm-dual} that $|\Omega_P| = |\Omega_{P^*}|$. Lemma~\ref{lem:poset-duality} shows that $\Omega_P$ and $\Omega_{P^*}$ are isomorphic as posets.

\begin{lemma} \label{lem:poset-duality}
    For a positroid $P$, $\Omega_P$ and $\Omega_{P^*}$ are isomorphic as posets, where the isomorphism is given by matroid duality.
\end{lemma}

\begin{proof}
    Let $E$ be the common ground-set of all matroids in $\Omega_P$. By construction, $|\Omega_P| = |\Omega_{P^*}|$, so it is sufficient to show that for all $\{N \leq M\} \subseteq \Omega_P$, $\{N^* \leq M^*\} \subseteq \Omega_{P^*}$. For $\{N \leq M\} \subseteq \Omega_P$, $\mathcal{B}(N) \subseteq \mathcal{B}(M)$. Thus,
    \[
        \mathcal{B}(N^*) = \{E \setminus B : B \in \mathcal{B}(N)\} \subseteq \{E \setminus B : B \in \mathcal{B}(M)\} = \mathcal{B}(M^*)
    \]
    and therefore $N^* \leq M^*$.
\end{proof}

Lemma~\ref{lem:weak-k-separate} shows that rank-preserving weak maps preserve $k$-separations. It immediately follows that a direct sum decomposition of a positroid $P$ induces a direct sum decomposition of all matroids contained in the positroid envelope class of $P$. This is formalized in the following.

\begin{lemma}[Lemma 2.48 in~\cite{quail2024positroid}] \label{lem:weak-k-separate}
    Let $M$ and $N$ be matroids such that $\mathbb{1} : M \to N$ is a rank-preserving weak map. Suppose that for $X \subseteq E$, $(X,E\setminus X)$ is a $k$-separation of $M$. Then $(X,E)$ is a $k$-separation of $N$. 
\end{lemma}

\begin{defi}
Let $M$ and $N$ be matroids on subsets of a totally ordered set $X$, with positroid envelopes $P$ and $Q$ respectively. Suppose that $E(M)$ and $E(N)$ are disjoint. Then we define the \emph{direct sum of envelope classes} $\Omega_P$ and $\Omega_Q$ as
\begin{equation*}
    \Omega_P \oplus \Omega_Q = \{M' \oplus N' : M' \in \Omega_P, N' \in \Omega_Q\}. 
\end{equation*}
\end{defi}

\begin{prop} [Proposition 2.52 in~\cite{quail2024positroid}] \label{prop:direct-sum-class-decomp}
    Let $P$ and $Q$ be positroids on disjoint, non-crossing subsets $X_1$ and $X_2$ of a totally ordered set $X$. Then
    \begin{equation*}
        \Omega_{P \oplus Q} = \Omega_P \oplus \Omega_Q.
    \end{equation*}
\end{prop}

For two posets $X_1$ and $X_2$, we equip the set $X_1 \times X_2$ with the partial order given by, for all $x_1,y_1 \in X_1$ and $x_2,y_2 \in X_2$, $(x_1,y_1) \leq (x_2,y_2)$ if $x_1 \leq x_2$ and $y_1 \leq y_2$. Lemma~\ref{lem:poset-direct-sum} shows that if a positroid envelope class decomposes under direct sums then it is isomorphic as a poset to the product of its positroid envelope class direct summands.

\begin{lemma} \label{lem:poset-direct-sum}
    Let $P$ and $Q$ be positroids on disjoint, non-crossing subsets of a totally ordered set. Then, $\Omega_{P \oplus Q}$ and $\Omega_P \times \Omega_Q$ are isomorphic as posets.
\end{lemma}

\begin{proof}
    We define the map $\varphi: \Omega_{P \oplus Q} \to \Omega_P \times \Omega_Q$ given by, for all $M \oplus N \in \Omega_{P \oplus Q}$, $\varphi(M \oplus N) = (M,N)$. This map is a bijection, so it remains to show that $\varphi$ and $\varphi^{-1}$ are order-preserving. Suppose that $M \oplus N, K \oplus L \in \Omega_{P \oplus Q}$ such that $K \oplus L \leq M \oplus N$, then $\mathcal{B}(K \oplus L) \subseteq \mathcal{B}(M \oplus N)$. Let $B_K \in \mathcal{B}(K)$ and $B_L \in \mathcal{B}(L)$ be arbitrary bases. Then, $B_K \oplus B_L \in \mathcal{B}(K \oplus L) \subseteq \mathcal{B}(M \oplus N)$. Furthermore, as $B_K \subseteq E(K) = E(M)$ and $B_L \subseteq E(L) = E(N)$, $B_K \in \mathcal{B}(M \oplus N | E(M)) = \mathcal{B}(M)$ and $B_L \in \mathcal{B}(M \oplus N | E(N)) = \mathcal{B}(N)$. Thus, $\mathcal{B}(K) \subseteq \mathcal{B}(M)$ and $\mathcal{B}(L) \subseteq \mathcal{B}(N)$. Therefore, $K \leq M$ and $L \leq N$, hence $(K,L) \leq (M,N)$.

    Suppose instead that $(K,L), (M,N) \in \Omega_P \times \Omega_Q$ such that $(K,L) \leq (M,N)$. Then, $K \leq M$ and $L \leq N$, so $\mathcal{B}(K) \subseteq \mathcal{B}(M)$ and $\mathcal{B}(L) \subseteq \mathcal{B}(N)$. Therefore, 
    \begin{align*}
        \mathcal{B}(K \oplus L) &= \{B_K \cup B_L : B_K \in \mathcal{B}(K), B_L \in \mathcal{B}(L)\}\\
        &\subseteq \{B_M \cup B_N : B_M \in \mathcal{B}(M), B_N \in \mathcal{B}(N)\} = \mathcal{B}(M \oplus N),
    \end{align*}
    hence $K \oplus L \leq M \oplus N$.
\end{proof}

We now build towards defining a \emph{$2$-sum of positroid envelope classes}, and a corresponding $2$-sum decomposition of positroid envelope classes. To do so, we use Lemmas~\ref{lem:circ-min} and~\ref{lem:circ-build} to prove that if $M$ has positroid envelope $P$, then a $2$-sum decomposition of $P$ induces a $2$-sum decomposition of $M$.

\begin{lemma}[Lemma 8.3.2 in~\cite{oxley2006matroid}] \label{lem:circ-min}
    Let $M$ be a matroid and $X \subseteq E$ such that $(X,E \setminus X)$ is a $2$-separation of $M$. Suppose that $C_1$ and $C_2$ are circuits of $M$ each of which meets both $X$ and $E \setminus X$. Then $C_1 \cap X$ is not a proper subset of $C_2 \cap X$.
\end{lemma}

\begin{lemma}[Lemma 8.3.3 in~\cite{oxley2006matroid}] \label{lem:circ-build}
    Let $M$ be a matroid and $X \subseteq E$ such that $(X,E \setminus X)$ is a $2$-separation of $M$. Let $Y_1$ and $Y_2$ be non-empty subsets of $X$ and $E \setminus X$, respectively. Suppose that $M$ has circuits $C_1$ and $C_2$ with $C_1 \cap X = Y_1$ and $C_2 \cap (E \setminus X) = Y_2$ such that $C_1 \cap (E \setminus X)$ and $C_2 \cap X$ are non-empty. Then $Y_1 \cup Y_2$ is a circuit of $M$.
\end{lemma}

\begin{lemma} \label{lem:2sum-weak-maps}
    Let $M$ and $N$ be matroids on subsets of a totally ordered set $Y = \{i_1 < i_2 < \cdots i_n\}$ such that $\mathbb{1}: M \to N$ is a rank-preserving weak map and $\pi^:_M = \pi^:_N$. Suppose that $M = M_1 \oplus_2 M_2$ such that
    \begin{itemize}
        \item[(i)] $\lvert E(M_1) \rvert \geq 3$, $\lvert E(M_2) \rvert \geq 3$, and $E(M_1) \cap E(M_2) = \{i_e\}$;
        \item[(ii)] $E(M_1)$ and $E(M_2)$ are non-crossing subsets of $Y$; and
        \item[(iii)] $(E(M_1) \setminus E(M_2),E(M_2) \setminus E(M_1))$ is not a $1$-separation of $M$ or $N$.
    \end{itemize}
    Then there exist matroids $N_1$ and $N_2$ such that $\pi^:_{N_1} = \pi^:_{M_1}$, $\pi^:_{N_2} = \pi^:_{M_2}$, and $N = N_1 \oplus_2 N_2$.
\end{lemma}

\begin{proof}
    Let $X = E(M_1) \setminus E(M_2)$ and $E = (E(M_1) \cup E(M_2)) \setminus \{i_e\}$. By (i), $\min \{\lvert X \rvert, \lvert E \setminus X \rvert \} \geq 2$. Furthermore,
    \begin{align*}
        \lambda_M(X) &= r_M(E(M_1) \setminus E(M_2)) + r_M(E(M_2) \setminus E(M_1)) - r_M(M)\\
        &= r_{M_1}(M_1) + r_{M_2}(M_2) - \big(r_{M_1}(M_1) + r_{M_2}(M_2) - 1\big)\\
        &= 1 < 2.
    \end{align*}
    Therefore, $(X,E \setminus X)$ is a $2$-separation of $M$. By Lemma~\ref{lem:weak-k-separate}, $(X,E \setminus X)$ is a $2$-separation of $N$. Let
    \begin{equation*}
        \mathcal{C}_1 = \mathcal{C}(N|X) \cup \{ (C \cap X) \cup i_e : C \in \mathcal{C}(M), C \cap X \neq \emptyset, C \cap (E \setminus X) \neq \emptyset \}
    \end{equation*}
    and
    \begin{align*}
        \mathcal{C}_2 = &\mathcal{C}(N|(E \setminus X)) \cup\\
        &\{(C \cap (E \setminus X)) \cup i_e: C \in \mathcal{C}(M), C \cap X \neq \emptyset, C \cap (E \setminus X) \neq \emptyset\}.
    \end{align*}

    By Lemmas~\ref{lem:circ-min} and~\ref{lem:circ-build}, $\mathcal{C}_1$ and $\mathcal{C}_2$ satisfy the circuit axioms for a matroid. Hence, $\mathcal{C}_1$ defines a matroid on $X \cup i_e$ that we denote by $N_1$ and $\mathcal{C}_2$ defines a matroid on $(E \setminus X) \cup i_e$ that we denote by $N_2$. It follows that $E(M_1) = E(N_1), E(M_2) = E(N_2)$, and $N = N_1 \oplus_2 N_2$.

    It remains to show that $\pi^:_{N_1} = \pi^:_{M_1}$ and $\pi^:_{N_2} = \pi^:_{M_2}$. By (ii), $E(M_1) = E(N_1)$ and $E(M_1) = E(N_2)$ are non-crossing subsets of $Y$. By (iii), $(E(M_1) \setminus E(M_2),E(M_2) \setminus E(M_1))$ is not a $1$-separator of $M$ or $N$, hence $i_e$ is neither a loop nor a coloop in $M$ or $N$. Therefore, by Lemma~\ref{dec-perm-2sum}, $\pi^:_{N_1} = \pi^:_{M_1}$ and $\pi^:_{N_2} = \pi^:_{M_2}$.
\end{proof}

It immediately follows that for an ordered matroid $M$ with positroid envelope $P$, a $2$-sum decomposition $P = P_1 \oplus_2 P_2$ into non-crossing positroids induces a $2$-sum decomposition $M = M_1 \oplus_2 M_2$ such that $\pi^:_{M_1} = \pi^:_{P_1}$ and $\pi^:_{M_2} = \pi^:_{P_2}$. We use the following lemma when $2$-sum decomposing positroid envelope classes.

\begin{lemma} \label{lem:not-1-sep}
    Let $P$ be a positroid on a subset of a totally ordered set $X$. Suppose $P = P_1 \oplus_2 P_2$, such that $P_1$ and $P_2$ are positroids on non-crossing subsets of $X$. Then, for any matroid $M \in \Omega_P$, $(E(P_1) \setminus E(P_2), E(P_2) \setminus E(P_1))$ is not a $1$-separation of $M$.
\end{lemma}

\begin{proof}
    Let $M_1 = M|(E(P_1) \setminus E(P_2))$ and $M_2 = M|(E(P_2) \setminus E(P_1))$. Suppose by way of contradiction that there exists $M \in \Omega_P$ such that $(E(P_1) \setminus E(P_2), E(P_2) \setminus E(P_1))$ is a $1$-separation of $M$. Then, $M = M_1 \oplus M_2$. By Remark~\ref{rmk:dec-perm-direct-sum} and Lemma~\ref{dec-perm-2sum}, we have
    \begin{equation*}
        \pi^:_M = \pi^:_{M_1 \oplus M_2} = \pi^:_{M_1} \sqcup \pi^:_{M_2} \neq \pi^:_{P_1 \oplus_2 P_2} = \pi^:_P,
    \end{equation*}
    which contradicts $\pi^:_M = \pi^:_P$. Therefore, $(E(P_1) \setminus E(P_2), E(P_2) \setminus E(P_1))$ is not a $1$-separation of $M$.
\end{proof}

We now have all the tools we need to define \emph{$2$-sums of positroid envelope classes} and present $2$-sum decompositions of positroid envelope classes, which are analogues of $2$-sums of positroids and $2$-sum decompositions of positroids respectively.

\begin{defi}
    Let $P$ and $Q$ be positroids on non-crossing subsets of a totally ordered set $X$, such that $|E(P)| \geq 2, |E(Q)| \geq 2$, and $E(P) \cap E(Q) = \{e\}$. Then we define the \emph{$2$-sum of positroid envelope classes} $\Omega_P$ and $\Omega_Q$ as
    \begin{equation*}
        \Omega_P \oplus_2 \Omega_Q := \{M \oplus_2 N : M \in \Omega_P, N \in \Omega_Q\}. 
    \end{equation*}
\end{defi}

\begin{prop} \label{prop:2sum-class-decomp}
    Let $P$ and $Q$ be $2$-connected positroids on non-crossing subsets of a totally ordered set $X$, such that $\lvert E(P) \rvert \geq 3$, $\lvert E(Q) \rvert \geq 3$, and $E(P) \cap E(Q) = \{e\}$. Then
    \begin{equation*}
        \Omega_{P \oplus_2 Q} = \Omega_P \oplus_2 \Omega_Q.
    \end{equation*}
\end{prop}

\begin{proof}
    Let $M_1 \in \Omega_P$ and $M_2 \in \Omega_Q$. Then, $\pi^:_{M_1} = \pi^:_P$, $\pi^:_{M_2} = \pi^:_Q$, and $E(M_1)$ and $E(M_2)$ are non-crossing subsets of $X$. By Lemma~\ref{dec-perm-2sum}, $\pi^:_{M_1 \oplus_2 M_2} = \pi^:_{P \oplus_2 Q}$, thus $M_1 \oplus_2 M_2 \in \Omega_{P \oplus_2 Q}$. As $M_1 \in \Omega_P$ and $M_2 \in \Omega_Q$ are arbitrary, it follows that $\Omega_P \oplus_2 \Omega_Q \subseteq \Omega_{P \oplus_2 Q}$.

    Let $M \in \Omega_{P \oplus_2 Q}$. By Lemma~\ref{lem:not-1-sep}, $(E(P) \setminus E(Q),E(Q) \setminus E(P))$ is not a $1$-separation of $M$. Then, by Lemma~\ref{lem:2sum-weak-maps}, there exist ordered matroids $M_1$ and $M_2$, such that $M = M_1 \oplus_2 M_2$, $\pi^:_{M_1} = \pi^:_P$, and $\pi^:_{M_2} = \pi^:_Q$. Thus, $M \in \Omega_P \oplus_2 \Omega_Q$. As $M \in \Omega_{P \oplus_2 Q}$ is arbitrary, it follows that $\Omega_{P \oplus_2 Q} \subseteq \Omega_P \oplus_2 \Omega_Q$. Therefore, $\Omega_{P \oplus_2 Q} = \Omega_P \oplus_2 \Omega_Q$.
\end{proof}

Taking Theorem~\ref{thm:pos-canon-tree-iff} together with repeated applications of Proposition~\ref{prop:2sum-class-decomp}, we can define a canonical tree decomposition for a positroid envelope class $\Omega_P$ corresponding to a $2$-connected positroid $P$.

\begin{defi}
    Let $P$ be a $2$-connected positroid, and let $T$ be a canonical tree decomposition of $P$ as described in Theorem~\ref{thm:pos-canon-tree-iff}. Then, we define a \emph{canonical tree decomposition of $\Omega_P$} as the tree $T_{\Omega}$ obtained from $T$ by relabelling each vertex $Q$ by $\Omega_Q$.
\end{defi}

We use Propositions~\ref{prop:2sum-minor} and \ref{prop:minor-express}, Corollary~\ref{cor:weak-map-1-separate}, and Theorem~\ref{thm:weak-map-minor} to show how rank-preserving weak maps interact with $2$-sums in Lemma~\ref{lem:2sum-weak-map}.

\begin{prop}[Proposition 7.1.21 in \cite{oxley2006matroid}] \label{prop:2sum-minor}
    Both $M$ and $N$ are isomorphic to minors of $M \oplus_2 N$.
\end{prop}

\begin{defi}[Definition 5.6 in \cite{lucas1975weak}]
   Let $N = (M \setminus X)/Y$ be a minor of the matroid $M$. We call $(M \setminus X)/Y$ a \emph{proper expression} of $N$ if given any sequence of one element operations in which the elements of $X$ are deleted and the elements of $Y$ are contracted then at no stage is a loop contracted or a coloop deleted. 
\end{defi}

\begin{prop}[Proposition 5.7 in \cite{lucas1975weak}] \label{prop:minor-express}
    Any minor can be properly expressed.
\end{prop}

\begin{thm}[Theorem 5.8 in \cite{lucas1975weak}] \label{thm:weak-map-minor}
    Let $M$ and $N$ be matroids such that $\mathbb{1}: M \to N$ is a rank-preserving weak map. Let $N'$ be a minor of $N$ properly expressed by $N' = (N \setminus X)/Y$ and let $M' = (M \setminus X)/Y$ be the corresponding minor of $M$. Then $\mathbb{1} : M' \to N'$ is a rank-preserving weak map.
\end{thm}

\begin{cor}[Corollary 5.3 in \cite{lucas1975weak}] \label{cor:weak-map-1-separate}
    Rank-preserving weak maps between matroids on the same finite ground-set preserve $1$-separations.
\end{cor}

\begin{lemma} \label{lem:2sum-weak-map}
    Let $M_1, M_2, N_1,$ and $N_2$ be matroids such that
    \begin{itemize}
        \item[(i)] $|E(M_1)| = |E(N_1)| \geq 2$, $|E(M_2)| = |E(N_2)| \geq 2$,
        \item[(ii)] $E(M_1) \cap E(M_2) = \{e\} = E(N_1) \cap E(N_2)$,
        \item[(iii)] $e$ is not a $1$-separator of $M_1,M_2,N_1,$ or $N_2$.
    \end{itemize}
    Then, $\mathbb{1} : M_1 \oplus_2 M_2 \to N_1 \oplus_2 N_2$ is a rank-preserving weak map if and only if $\mathbb{1} : M_1 \to N_1$ and $\mathbb{1} : M_2 \to N_2$ are both rank-preserving weak maps. 
\end{lemma}

\begin{proof}
    Suppose that $\mathbb{1} : M_1 \to N_1$ and $\mathbb{1} : M_2 \to N_2$ are both rank-preserving weak maps. By construction, $E(M_1 \oplus_2 M_2) = E(N_1 \oplus_2 N_2)$ and 
    \begin{align*}
        r_{M_1 \oplus_2 M_2}(M_1 \oplus_2 M_2) &= r_{M_1}(M_1) + r_{M_2}(M_2) - 1\\
        &= r_{N_1}(N_1) + r_{N_2}(N_2) - 1 = r_{N_1 \oplus_2 N_2}(N_1 \oplus_2 N_2),
    \end{align*}
    so it is sufficient to show that $\mathbb{1} : M_1 \oplus_2 M_2 \to N_1 \oplus_2 N_2$ is a weak map. We will prove the contrapositive, if $\mathbb{1} : M_1 \oplus_2 M_2 \to N_1 \oplus_2 N_2$ is not a weak map, then $\mathbb{1} : M_1 \to N_1$ or $\mathbb{1} : M_2 \to N_2$ is not a weak map. 
    
    Suppose that $\mathbb{1} : M_1 \oplus_2 M_2 \to N_1 \oplus_2 N_2$ is not a weak map. Then, there exists some $I \in \mathcal{I}(N_1 \oplus_2 N_2)$ such that $I \notin \mathcal{I}(M_1 \oplus_2 M_2)$. Therefore, there exists a circuit $C \in \mathcal{C}(M_1 \oplus_2 M_2)$ such that $C \subseteq I$. Recall that
    \begin{align*}
        \mathcal{C}(M_1 \oplus_2 M_2) = \mathcal{C}&(M_1 \setminus e) \cup \mathcal{C}(M_2 \setminus e) \cup\\ \{&(C_1 \cup C_2) \setminus \{e\} : e \in C_1 \in \mathcal{C}(M_1), e \in C_2 \in \mathcal{C}(M_2)\}. 
    \end{align*}
    Suppose further that $C \in \mathcal{C}(M_1 \setminus e)$. Then $C \in \mathcal{I}(N_1)$, but $C \notin \mathcal{I}(M_1)$, so $\mathbb{1} : M_1 \to N_1$ is not a weak map. By symmetry, if $C \in \mathcal{C}(M_2 \setminus e)$, then $\mathbb{1} : M_2 \to N_2$ is not a weak map. Suppose instead that $C \in \{(C_1 \cup C_2) \setminus \{e\} : e \in C_1 \in \mathcal{C}(M_1), e \in C_2 \in \mathcal{C}(M_2)\}$. Then, there exist circuits $C_1 \in \mathcal{C}(M_1)$, $C_2 \in \mathcal{C}(M_2)$ such that $C_1 \cap C_2 = \{e\}$ and $C = (C_1 \cup C_2) \setminus \{e\}$. As $C \in \mathcal{I}(N_1 \oplus_2 N_2)$, it follows that $C_1 \in \mathcal{I}(N_1)$ or $C_2 \in \mathcal{I}(N_2)$. Therefore, $\mathbb{1} : M_1 \to N_1$ or $\mathbb{1} : M_2 \to N_2$ is not a weak map.

    Now suppose that $\mathbb{1} : M_1 \oplus_2 M_2 \to N_1 \oplus_2 N_2$ is a rank-preserving weak map. By Proposition~\ref{prop:2sum-minor} and Proposition~\ref{prop:minor-express}, there exists a properly expressed minor $N' = ((N_1 \oplus_2 N_2) \setminus X)/Y = N_1 \oplus_2 ((N_2 \setminus X)/Y)$ of $N_1 \oplus_2 N_2$ such that $(N_2 \setminus X)/Y \cong U^1_2$ and $N' \cong N_1$, where $X,Y \subset E(N_2)$ and $N'$ is obtained from $N_1$ by relabelling $e$. Then, by Theorem~\ref{thm:weak-map-minor}, there exists a minor $M' = ((M_1 \oplus_2 M_2) \setminus X)/Y = M_1 \oplus_2 ((M_2 \setminus X)/Y)$ such that $\mathbb{1} : M' \to N'$ is a rank-preserving weak map. By Corollary~\ref{cor:weak-map-1-separate}, $e$ is neither a $1$-separator of $M_1$ nor of $(M_2 \setminus X)/Y$, so $(M_2 \setminus X)/Y$ is $2$-connected. The uniform matroid $U^1_2$ is, up to isomorphism, the unique $2$-connected matroid on two elements, so $(M_2 \setminus X)/Y \cong U^1_2$. Thus, $M' = M_1 \oplus_2 ((M_2 \setminus X)/Y) \cong M_1$. As both $M_1$ and $N_1$ are obtained from $M'$ and $N'$ respectively by relabelling $e$, it follows that $\mathbb{1} : M_1 \to N_1$ is a rank-preserving weak map. By symmetry, $\mathbb{1} : M_2 \to N_2$ is a rank-preserving weak map.
\end{proof}

We apply Lemma~\ref{lem:2sum-weak-map} to show in Lemma~\ref{lem:poset-2sum} that if a positroid envelope class decomposes under $2$-sums, then the positroid envelope class is isomorphic as a poset to the product of its positroid envelope class $2$-summands.

\begin{lemma} \label{lem:poset-2sum}
    Let $P$ and $Q$ be $2$-connected positroids on non-crossing subsets of a totally ordered set such that $E(P) \cap E(Q) = \{e\}$, $|E(P)| \geq 2$, and $|E(Q)| \geq 2$. Then, $\Omega_{P \oplus_2 Q}$ and $\Omega_P \times \Omega_Q$ are isomorphic as posets.
\end{lemma}

\begin{proof}
    We define the map $\varphi : \Omega_{P \oplus_2 Q} \to \Omega_P \times \Omega_Q$ given by, for all $M \oplus_2 N \in \Omega_{P \oplus_2 Q},$ $\varphi(M \oplus_2 N) = (M,N)$. This map is a bijection, so it remains to show that $\varphi$ and $\varphi^{-1}$ are order-preserving. Suppose that $M \oplus_2 N, K \oplus_2 L \in \Omega_{P \oplus_2 Q}$ such that $K \oplus_2 L \leq M \oplus_2 N$. Then, $\mathbb{1} : M \oplus_2 N \to K \oplus_2 L$ is a rank-preserving weak map. By Lemma~\ref{lem:2sum-weak-map}, $\mathbb{1} : M \to K$ and $\mathbb{1} : N \to L$ are both rank-preserving weak maps, hence $K \leq M$ and $L \leq N$. Therefore, $\varphi(K \oplus_2 L) = (K,L) \leq (M,N) = \varphi(M \oplus_2 N)$.

    Suppose that $(M,N),(K,L) \in \Omega_P \times \Omega_Q$ such that $(K,L) \leq (M,N)$. Then, $K \leq M$ and $L \leq N$, so $\mathbb{1} : M \to K$ and $\mathbb{1} : N \to L$ are both rank-preserving weak maps. By Lemma~\ref{lem:2sum-weak-map}, $\mathbb{1} : M \oplus_2 N \to K \oplus_2 L$ is a rank-preserving weak map. Therefore, $\varphi^{-1}(K,L) = K \oplus_2 L \leq M \oplus_2 N = \varphi^{-1}(M,N).$ 
\end{proof}

Proposition~\ref{prop:semilattices-direct-sums-2sums} immediately follows from Lemmas~\ref{lem:poset-direct-sum} and \ref{lem:poset-2sum}. Note that the posets $\Omega_P$ and $\Omega_Q$ are both join-semilattices, meet-semilattices, or lattices if and only if $\Omega_P \times \Omega_Q$ is a join-semilattice, meet-semilattice, or lattice, respectively. 

\begin{prop} \label{prop:semilattices-direct-sums-2sums}
    Let $P$ and $Q$ be positroids on non-crossing subsets  of a totally ordered set.
    \begin{enumerate}
        \item Suppose that $E(P)$ and $E(Q)$ are disjoint, then $\Omega_P$ and $\Omega_Q$ are both join-semilattices/meet-semilattices/lattices if and only if $\Omega_{P \oplus Q}$ is a join-semilattice/meet-semilattice/lattice.
        \item Suppose that $|E(P)| \geq 2, |E(Q)| \geq 2,$ and $E(P) \cap E(Q) = \{e\}$, then $\Omega_P$ and $\Omega_Q$ are both join-semilattices/meet-semilattices/lattices if and only if $\Omega_{P \oplus_2 Q}$ is a join-semilattice/meet-semilattice/lattice.
    \end{enumerate}
\end{prop}

By applying Proposition~\ref{prop:semilattices-direct-sums-2sums}, we characterize in Theorem~\ref{thm:semilattices-rank-corank-2} all positroid envelope classes as join-semilattices whose positroids decompose under direct sums and $2$-sums into positroids of rank or corank at most $2$.

\begin{thm} \label{thm:semilattices-rank-corank-2}
    Let $P$ be a positroid that decomposes under direct sums and $2$-sums into a collection of positroids all of which have rank or corank at most $2$. Then $\Omega_P$ is a join-semilattice.
\end{thm}

\begin{proof}
    By Proposition~\ref{prop:semilattices-direct-sums-2sums}, it is sufficient to show that if $P$ is $2$-connected and has rank or corank at most $2$ then $\Omega_P$ is a join-semilattice. Suppose that $P$ has rank or corank at most $1$, then $P$ is $U^2_4$-free. By Theorem \ref{thm:pos-graphic}, $\Omega_P$ contains a single element and is therefore a join-semilattice.

    Now suppose that $P$ has rank-$2$. By assumption, $P$ is $2$-connected, so $P$ is loopless and therefore every matroid in $\Omega_P$ is loopless. Let $M,N \in \Omega_P$ be distinct ordered matroids, and without loss of generality let $E$ be the ground-set of all matroids in $\Omega_P$. We construct the following set
    \[
        \mathcal{C} = \left\{ {E \choose 2} \notin \mathcal{B}(M) \cup \mathcal{B}(N) \right\} \cup \left\{ {E \choose 3} : {E \choose 3} \not\supset {E \choose 2} \notin \mathcal{B}(M) \cup \mathcal{B}(N) \right\},
    \]
    and we will show that $\mathcal{C}$ satisfies the circuit axioms of a matroid. By construction, no element of $\mathcal{C}$ is a proper subset of any other element of $\mathcal{C}$. Let $C_1,C_2 \in \mathcal{C}$ such that there exists $e \in C_1 \cap C_2$. Suppose that $|C_1| = 2 = |C_2|$, then as both $M$ and $N$ are loopless, we obtain $(C_1 \cup C_2) \setminus \{e\} \in \mathcal{C}(M) \cap \mathcal{C}(N)$ and therefore
    \[
        (C_1 \cup C_2) \setminus \{e\} \in \left\{ {E \choose 2} \notin \mathcal{B}(M) \cup \mathcal{B}(N) \right\} \subseteq \mathcal{C}.
    \]
    Suppose instead that $|C_1| = 3$ or $|C_2| = 3$. Then either
    \[
         (C_1 \cup C_2) \setminus \{e\} \supset C \in \left\{ {E \choose 2} \notin \mathcal{B}(M) \cup \mathcal{B}(N) \right\},
    \]
    or there exists a three element subset $C_3 \subset (C_1 \cup C_2) \setminus \{e\}$ such that
    \[
        C_3 \in \left\{ {E \choose 3} : {E \choose 3} \not\supset {E \choose 2} \notin \mathcal{B}(M) \cup \mathcal{B}(N) \right\},
    \]
    thus $(C_1 \cup C_2) \setminus \{e\} \supseteq C \in \mathcal{C}$. Therefore, $\mathcal{C}$ is the set of circuits of some rank-$2$ matroid $K$ on $E$. The bases of $K$ are given by
    \[
        \mathcal{B}(K) = \left\{ {E \choose 2} \right\} \mathbin{\bigg\backslash} \left\{ {E \choose 2} \notin \mathcal{B}(M) \cup \mathcal{B}(N) \right\} = \mathcal{B}(M) \cup \mathcal{B}(N).
    \]
    It follows that $\mathbb{1} : K \to M$ and $\mathbb{1} : K \to N$ are both rank-preserving weak maps, $M,N \leq K$, and any matroid $L$ such that $M,N \leq L$ has the property that $K \leq L$. Therefore, $K$ is the join of $M$ and $N$ in $\Omega_P$. Furthermore, as $M$ and $N$ were arbitrary, $\Omega_P$ is a join-semilattice.
\end{proof}

\subsection{Positroid varieties} For any field $\mathbb{F}$, the \emph{Grassmannian} $\Gr(r,\mathbb{F}^n)$ is the collection of $r$-dimensional linear subspaces of $\mathbb{F}^n$. A point $V \in \Gr(r,\mathbb{F}^n)$ can be represented by an $r \times n$ full-rank matrix $A$ with entries over $\mathbb{F}$ whose row-span is $V$. As any such representative matrix realizes the same ordered matroid, we can associate to each point $ V \in \Gr(r,\mathbb{F}^n)$ an ordered matroid $M_V$. The \emph{matroid stratum} over a field $\mathbb{F}$ of a rank-$r$ ordered matroid $M$ on $n$ elements is
\[
    S_M(\mathbb{F}) := \{V \in \Gr(r,\mathbb{F}^n) : M_V = M\}.
\]
Observe that $M$ is $\mathbb{F}$-realizable if and only if $S_M(\mathbb{F})$ is non-empty. We say that the matroid stratum associated to a matroid $M$ is binary or ternary if $M$ is binary or ternary respectively.

The totally nonnegative Grassmannian $\Gr^{\geq 0}(r,\mathbb{R}^n)$ consists of the points in the Grassmannian $\Gr(r,\mathbb{R}^n)$ that can be represented by matrices with all nonnegative maximal minors. The \emph{positroid cell} assocaited to a rank-$r$ ordered matroid $M$ on $n$ elements
\[
    S^{\geq 0}_M := S_M(\mathbb{R}) \cap \Gr^{\geq 0}(r,\mathbb{R}^n).
\]
An ordered matroid is a positroid if and only if its corresponding positroid cell is non-empty.

For a field $\mathbb{F}$ and a positrod $P$ of rank-$r$ on $n$ elements, the \emph{open positroid variety} $\Pi^{\circ}_{\mathbb{F}}(P)$ is defined as
\[
    \Pi^{\circ}_{\mathbb{F}}(P) := \{V \in \Gr(r,\mathbb{F}^n) : M_V \in \Omega_P\} = \bigsqcup_{M \in \Omega_P} S_M(\mathbb{F}),
\]
and the \emph{positroid variety} $\Pi_{\mathbb{F}}(P) = \overline{\Pi^{\circ}_{\mathbb{F}}(P)}$. The following result gives an equivalent description of the positroid varieties over $\mathbb{R}$.

\begin{cor}[Corollary 5.12 in~\cite{knutson2013positroid}] \label{cor:pos-variety}
    Let $P$ be a positroid. Then, as sets,
    \begin{equation*}
        \Pi_{\mathbb{R}}(P) = \overline{S_P} = \{V \in \Gr(k,\mathbb{R}^n) : I \notin \mathcal{I}(P) \Rightarrow \Delta_I(V) = 0\}.
    \end{equation*}
\end{cor}

It immediately follows that
\[
    \Pi_{\mathbb{R}}(P) = \bigsqcup_{\mathbb{1} : P \xrightarrow{\text{rp}} M} S_M(\mathbb{R}),
\]
where the disjoint union is taken over all ordered matroids $M$ for which $\mathbb{1} : P \to M$ is a rank-preserving weak map.

\subsection{Binary positroids and their envelope classes}

A matroid is called \emph{series-parallel} if it can be represented by a series-parallel graph. The following result characterizes the $2$-connected series-parallel positroids.

\begin{cor}[Corollary 6.4 in~\cite{speyer2021positive}]
    A $2$-connected positroid is series-parallel if and only if it is $U^2_4$-free.
\end{cor}

Binary positroids have several equivalent simple characterizations, one of which is the direct sum of series-parallel positroids, summarized in the following theorem. Since binary positroids are regular, they are also ternary. Our main results in Section~\ref{sec:main} build on this result and cover the remaining characterization of the ternary positroids.

\begin{thm} [Theorem 4.4 in~\cite{quail2024positroid}] \label{thm:pos-graphic}
    The following are equivalent for a positroid $P$.
    \begin{itemize}
        \item[(i)] $P$ is binary,
        \item[(ii)] $P$ is graphic,
        \item[(iii)] $P$ is regular,
        \item[(iv)] $P$ is a direct sum of series-parallel matroids,
        \item[(v)] $P$ is the unique matroid contained in $\Omega_P$,
        \item[(vi)] for every field $\mathbb{F}$, $\Pi^{\circ}_{\mathbb{F}}(P) = S_P(\mathbb{F})$,
        \item[(vii)] $\Pi_{\mathbb{R}}(P)$ is equal to the disjoint union of graphic matroid strata,
    \end{itemize} 
\end{thm}

\section{Main results}\label{sec:main}

\subsection{Ternary and quaternary positroids}

In this subsection, we characterize the ternary and quaternary matroids with positroid orderings by their sets of minimal forbidden minors. We begin by studying the structure of the ternary positroids. By Theorems~\ref{thm:pos-direct-sum} and \ref{thm:pos-canon-tree-iff}, every positroid can be decomposed into direct sums and $2$-sums of positroids that are $3$-connected, circuits, or cocircuits. For any field $\mathbb{F}$, the direct sum or $2$-sum of $\mathbb{F}$-linear matroids is itself $\mathbb{F}$-linear. So, in characterizing the ternary positroids, it is sufficient to characterize the ordered $3$-connected matroids, circuits, and cocircuits that are ternary positroids. Every ordered circuit and cocircuit is a graphic positroid, which are all regular and therefore ternary. Furthermore, any $3$-connected matroid on fewer than four elements is a circuit or a cocircuit, so it remains to characterize the $3$-connected ternary positroids on at least four elements. 

To do so, we first consider the class of matroids called \emph{gammoids}, which contains all positroids as shown in~\cite{chidiac2021positroids}. We direct the reader to~\cite{oxley2006matroid} for a formal introduction to gammoids. Here, we will make use of the following result characterizing the $3$-connected ternary gammoids on at least four elements to constrain the class of $3$-connected ternary positroids. Note that $\mathcal{W}^r$ is the rank-$r$ whirl, which we will define shortly.

\begin{cor}[Corollary 4.3 in~\cite{oxley1987characterization}] \label{cor:tern-3connect}
    Let $M$ be a $3$-connected ternary gammoid having at least $4$ elements. Then $M \cong \mathcal{W}^r$ for some $r \geq 2$.
\end{cor}

Now, we define two important graphs for this section: $\mathcal{W}_r$ and $\mathcal{N}_r$. Let $\mathcal{W}_r$ be the wheel graph on $r+1$ vertices. A wheel graph is formed by adding a dominating vertex to the cycle graph on $r$ vertices. We call the edges incident to this dominating vertex the \emph{spokes} of the wheel and the edge set $R$ of the original cycle the \emph{rim}. The \emph{whirl} is the matroid $\mathcal{W}^r$ with $E(\mathcal{W}^r)=E(\mathcal{W}_r)$, and $\mathcal{B}(\mathcal{W}^r)= \mathcal{B}(M(\mathcal{W}_r))\cup \{  R\}$. The set $R$ is both a circuit and a hyperplane of $M(\mathcal{W}_r)$, and the formation of the new matroid $\mathcal{W}^r$ by adding $R$ to $\mathcal{B}(M(\mathcal{W}_r))$ is an example of a \emph{circuit-hyperplane relaxation}. We let $\mathcal{N}_r$ be the graph formed by a cycle $C_r$ of length $r$, joined at a vertex to its dual. For $r > 2$, the corresponding matroid $M( \mathcal{N}_r)$ has a unique circuit-hyperplane (the cycle of length $r$), and we denote its circuit-hyperplane relaxation by $\mathcal{N}^r$. When $r = 2$, $M(\mathcal{N}_2)$ has exactly two circuit-hyperplanes and the relaxation of either yields isomorphic matroids, so we may unambiguously denote them by $\mathcal{N}^2$. Both graphs  $\mathcal{W}_r$ and $\mathcal{N}_r$ are shown in Figure~\ref{fig:wheel-graph}.
In Lemma~\ref{lem:whirl-pos}, we show that every wheel on at least three vertices has an edge-ordering so that its corresponding whirl is a positroid. We use Propositions~\ref{prop:circ-hyp-relax-circuits} and~\ref{prop:circ-hyp-relax-cocircuits} about circuit-hyperplane relaxations.

\begin{prop}[Proposition 1.5.14 in \cite{oxley2006matroid}] \label{prop:circ-hyp-relax-circuits}
    Let $M$ be a matroid having a subset $X$ that is both a circuit and a hyperplane. Let $\mathcal{B'} = \mathcal{B}(M) \cup \{X\}$ (referred to as the relaxation of the circuit-hyperplane $X$ of $M$). Then $\mathcal{B'}$ is the set of bases of a matroid $M'$ on $E(M)$. Moreover,
    \begin{equation*}
        \mathcal{C}(M') = (\mathcal{C}(M) \setminus \{X\}) \cup \{X \cup e: e \in E(M) \setminus X\}.
    \end{equation*}
\end{prop}

\begin{prop}[Proposition 2.1.7 in \cite{oxley2006matroid}] \label{prop:circ-hyp-relax-cocircuits}
    If $M'$ is obtained from $M$ by relaxing a circuit-hyperplane $X$ of $M$, then $(M')^*$ can be obtained from $M^*$ by relaxing the circuit-hyperplane $E(M) \setminus X$ of $M^*$.
\end{prop}

\begin{figure}
    \centering
    \begin{tikzpicture}[scale=.9]
        \draw[fill=black] (0,0) circle (3pt);
        \draw[fill=black] (0:2) circle (3pt);
        \draw[fill=black] (45:2) circle (3pt);
        \draw[fill=black] (90:2) circle (3pt);
        \draw[fill=black] (135:2) circle (3pt);
        \draw[fill=black] (180:2) circle (3pt);
        \draw[fill=black] (195:2) circle (1pt);
        \draw[fill=black] (210:2) circle (1pt);
        \draw[fill=black] (225:2) circle (1pt);
        \draw[fill=black] (240:2) circle (1pt);
        \draw[fill=black] (255:2) circle (1pt);
        \draw[fill=black] (270:2) circle (3pt);
        \draw[fill=black] (315:2) circle (3pt);

        \draw[thick] (0,0) -- (0:2) node[above,midway] {7};
        \draw[thick] (0,0) -- (45:2) node[above,midway] {5};
        \draw[thick] (0,0) -- (90:2) node[left,midway] {3};
        \draw[thick] (0,0) -- (135:2) node[left,midway] {1};
        \draw[thick] (0,0) -- (180:2) node[below,midway] {2r-1};
        \draw[thick] (0,0) -- (270:2) node[right,midway] {\rotatebox{-90}{11}};
        \draw[thick] (0,0) -- (315:2) node[right,midway] {\rotatebox{-45}{9}};

        \draw[thick] (0:2) arc (0:45:2) node[right,midway] {6};
        \draw[thick] (45:2) arc (45:90:2) node[above,midway] {4};
        \draw[thick] (90:2) arc (90:135:2) node[above,midway] {2};
        \draw[thick] (135:2) arc (135:180:2) node[left,midway] {2r};
        \draw[thick] (270:2) arc (270:315:2) node[below,midway] {10};
        \draw[thick] (315:2) arc (315:360:2) node[right,midway] {8};

        \draw (-2,2) node{$\mathcal{W}_r$};

\begin{scope}[xshift=6cm]

        wheel graph
        \draw[fill=black] (0:2) circle (3pt);
        \draw[fill=black] (0:5) circle (3pt);
        \draw[fill=black] (45:2) circle (3pt);
        \draw[fill=black] (90:2) circle (3pt);
        \draw[fill=black] (135:2) circle (3pt);
        \draw[fill=black] (180:2) circle (3pt);
        \draw[fill=black] (195:2) circle (1pt);
        \draw[fill=black] (210:2) circle (1pt);
        \draw[fill=black] (225:2) circle (1pt);
        \draw[fill=black] (240:2) circle (1pt);
        \draw[fill=black] (255:2) circle (1pt);
        \draw[fill=black] (270:2) circle (3pt);
        \draw[fill=black] (315:2) circle (3pt);

        \draw[thick] (0:2) arc (0:45:2) node[right,midway] {6};
        \draw[thick] (45:2) arc (45:90:2) node[above,midway] {4};
        \draw[thick] (90:2) arc (90:135:2) node[above,midway] {2};
        \draw[thick] (135:2) arc (135:180:2) node[left,midway] {2r};
        \draw[thick] (270:2) arc (270:315:2) node[below,midway] {10};
        \draw[thick] (315:2) arc (315:360:2) node[right,midway] {8};

\draw [black!100,line width=1pt] (2,0) to [in=100, out=80] (5,0);
\draw [black!100,line width=1pt] (2,0) to [in=160, out=20] (5,0);
\draw [black!100,line width=1pt] (2,0) to [in=-100, out=-80] (5,0);
\draw [black!100,line width=1pt] (2,0) to [in=-160, out=-20] (5,0);

\draw[fill=black] (3.5,.15) circle (1pt);
\draw[fill=black] (3.5,0) circle (1pt);
\draw[fill=black] (3.5,-.15) circle (1pt);

\draw (3.5,1.1) node{$1$};
\draw (3.5,.55) node{$3$};
\draw (3.5,-1.1) node{$2r-1$};
\draw (3.5,-.55) node{$2r-3$};
\draw (-2,2) node{$\mathcal{N}_r$};

\end{scope}

    \end{tikzpicture}
    \caption{The wheel graph $\mathcal{W}_r$ and the graph $\mathcal{N}_r$, both with a positroid edge-ordering.}
    \label{fig:wheel-graph}
\end{figure}

\begin{lemma}\label{lem:whirl-pos}
The whirl $\mathcal{W}^r$, with $r \geq 2$, has a positroid edge-ordering. 
\end{lemma}

\begin{proof}
Without loss of generality, let $\mathcal{W}_r$ be the wheel graph with $E(\mathcal{W}_r) = [2r]$, where $r \geq 2$, with edge-ordering given as in Figure \ref{fig:wheel-graph}. Let $R = \{2,4,\ldots,2r-2,r\} \subset E(\mathcal{W}_r)$ be the rim of the wheel graph. Observe that $R$ is a circuit-hyperplane of $M(\mathcal{W}_r)$, and when $r \neq 3$, $R$ is the unique circuit-hyperplane of $M(\mathcal{W}_r)$. Let $\mathcal{W}^r$ be the rank-$r$ whirl obtained from $M(\mathcal{W}_r)$ by the circuit-hyperplane relaxation of $R$. 
By Propositions \ref{prop:circ-hyp-relax-circuits} and \ref{prop:circ-hyp-relax-cocircuits}, the circuits of $\mathcal{W}^r$ are
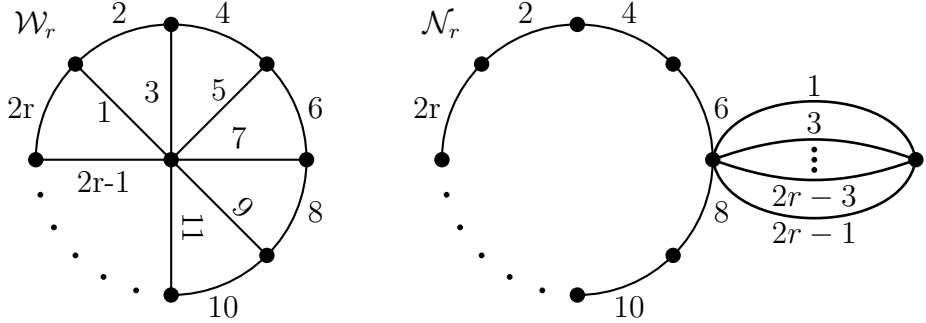
\[
    \mathcal{C}(\mathcal{W}^r) = (\mathcal{C}(M(\mathcal{W}_r)) \setminus \{R\}) \cup \{R \cup \{e\} : e \in E(\mathcal{W}^r) \setminus R\},
\]
and the cocircuits of $\mathcal{W}^r$ are
\[\mathcal{C}^*(\mathcal{W}^r) = (\mathcal{C}^*(M(\mathcal{W}_r)) \setminus \{E(\mathcal{W}^r) \setminus R\}) \cup \{(E(\mathcal{W}^r) \setminus R) \cup \{e\} : e \in R\}.
\]

Let $C \in \mathcal{C}(\mathcal{W}^r)$ and $C^*\in \mathcal{C}^*(\mathcal{W}^r)$ such that $C \cap C^* = \emptyset$. Note that if $C$ is of the form $R \cup \{e\}$, then there are no cocircuits disjoint from it. By duality, if $C^*$ is of the form $(E(\mathcal{W}^r) \setminus R) \cup \{e\}$, then there are no circuits disjoint from it. Therefore, $C \in \mathcal{C}(M(\mathcal{W}_r) \setminus \{R\})$ and $C^* \in \mathcal{C}^*(M(\mathcal{W}_r)) \setminus \{E(\mathcal{W}_r) \setminus R\}$. Every circuit $C$ in $\mathcal{W}^r$ corresponds to a cycle in $\mathcal{W}_r$ that includes the center vertex. Therefore, it has the form, $i,i+1,i+3,\dots,j-1,j$, where edges $i$ and $j$ are incident to the center vertex, and the others are on the rim. The spoke edges $i+2,i+4,\dots,j-2$ cannot be part of a bond that is disjoint from $C$, since all of their endpoints are on the cycle $C$ in $\mathcal{W}_r$. Therefore, $C^*\subseteq \{ j+1,j+2,\dots,i-1  \}$, hence $C$ and $C^*$ are disjoint, non-crossing subsets of $E(\mathcal{W}^r)$. By Theorem \ref{thm:pos-crossing}, $\mathcal{W}^r$ is a positroid.
\end{proof}

It follows from Corollary~\ref{cor:tern-3connect}, that any $3$-connected positroid that is neither a circuit nor a cocircuit is isomorphic to a rank-$r$ whirl for some $r \geq 2$. Therefore, any ternary positroid can be obtained from direct sums and $2$-sums from circuits, cocircuits, and matroids isomorphic to rank-$r$ whirls for $r \geq 2$. We will formalize this structural characterization later, in Theorem~\ref{thm:pos-tern-minor}. In Lemma~\ref{lem:whirl-pos}, we showed that whirls have positroid orderings by constructing an appropriate ordering on the edges of the wheel graph, shown in Figure~\ref{fig:wheel-graph}. In fact, up to a shift by $1$ and order-preserving isomorphism, this is the only such ordering. This restricts the possible decorated permutations of such a positroid to two options. We formalize this in Lemma~\ref{lem:whirl-ordering} and Corollary~\ref{cor:whirl-perm}. 

First, observe that the rank-$2$ whirl is isomorphic to the uniform matroid $U^2_4$. For a totally ordered set $E$, there is a unique rank-$r$ positroid on $E$ isomorphic to the uniform matroid $U^r_{|E|}$. Therefore, for $|E| = 4$, there is a unique positroid $P$ on $E$ such that $P \cong U^2_4 \cong \mathcal{W}^r$. We now consider positroids isomorphic to whirls of rank greater than two.

\begin{lemma}\label{lem:whirl-ordering}
    Let $P$ be a positroid on a totally ordered set $E$. Suppose that for some $r \geq 3$, $P$ is isomorphic to the rank-$r$ whirl $\mathcal{W}^r$. Then, $P$ and $P^*$ are the unique positroids on $E$ isomorphic to $\mathcal{W}^r$.
\end{lemma}

\begin{proof}
    As positroids are closed under order-preserving isomorphisms, without loss of generality we may consider $P$ to be a positroid on $[2r]$. Let $\mathcal{W}_r$ be an edge-ordered wheel graph whose rim we denote by $R$, as in the proof of Lemma~\ref{lem:whirl-pos}, such that $P$ is obtained from $M(\mathcal{W}_r)$ by the circuit-hyperplane relaxation of $R$.
    We note that the set of circuits of $P$ is 
    \[
        \mathcal{C}(M(\mathcal{W}_r)) \setminus \{R\} \cup \{R \cup e : e \in E \setminus R\},
    \]
    and the set of cocircuits of $P$ is
    \[
        \mathcal{C}^*(M(\mathcal{W}_r) \setminus \{E \setminus R\}) \cup \{(E \setminus R) \cup e : e \in R\}.
    \]
  
    By Theorem~\ref{thm:pos-crossing}, for any $C\in \mathcal{C}(P)$ and $C^*\in \mathcal{C}^*(P)$ such that $C \cap C^* = \emptyset$, we must have that $C$ and $C^*$ are non-crossing subsets of $E$. Consider a cocircuit $C^*$ formed by the edges incident to one non-center vertex in $\mathcal{W}_r$. The graph obtained from $\mathcal{W}_r$ by deleting $C^*$ does not contain the cycle formed by $R$. Furthermore, $\mathcal{W}_r \setminus C^*$ contains at least $3$ vertices and is 2-vertex-connected, hence by Proposition 4.1.7 in~\cite{oxley2006matroid}, $M(\mathcal{W}_r \setminus C^*) = M(\mathcal{W}_r) \setminus C^*$ is $2$-connected. This implies that any two elements in $E \setminus C^*$ appear together in some circuit contained in $\mathcal{C}(M(\mathcal{W}_r)) \setminus \{R\}$ and therefore in some circuit of $P$. Furthermore, this implies that $C^*$ must be a cyclic interval of $E$, or else we can find a disjoint cycle $C \in \mathcal{C}(P)$ for which $C$ and $C^*$ are crossing subsets of $E(P)$. By duality, the same holds for every triangle in $\mathcal{W}_r$. This forces $\mathcal{W}_r$, up to graph isomorphism, to have one of two edge-orderings: either the ordering shown in Figure~\ref{fig:wheel-graph}, with odd labels on the spokes and even labels on the rim, or a shift of this labeling by $1$ modulo $2r$, which has even labels on the spokes and odd labels on the rim.
\end{proof}

\begin{cor} \label{cor:whirl-perm}
    Let $P$ be a positroid on $[2r]$, such that $P\cong \mathcal{W}^r$, for some $r \geq 2$. Then
    \begin{equation*}
        \pi^:_P = (1,3,\ldots,{2r-1})({2r},{2r-2},\ldots,2)
    \end{equation*}
    \hspace{175pt} or
    \begin{equation*}
        \pi^:_P = (2,4,\ldots,2r)({2r-1},{2r-3},\ldots,1).
    \end{equation*}
\end{cor}

We use the following characterization of ternary gammoids when characterizing the ternary positroids in Theorem~\ref{thm:pos-tern-minor}. Note that a matroid is called \emph{near-regular} if it is realizable over all fields except possibly $\mathbb{F}_2$.

\begin{thm}[\cite{oxley1987characterization}, \cite{oxley2006matroid}] \label{thm:tfae:tern-gammoids}
    Let $M$ be a matroid. The following are equivalent:
    \begin{itemize}
        \item[(i)] $M$ is a ternary gammoid,
        \item[(ii)] $M$ is a near-regular gammoid \cite[Theorems 14.7.5,14.7.7,14.7.9]{oxley2006matroid},
        \item[(iii)] $M$ has no minor isomorphic to any of the matroids $U^2_5$, $U^3_5$, $M(K_4)$, $P_7$, or $P^*_7$ \cite[Theorem 4.1]{oxley1987characterization}.
    \end{itemize}
\end{thm}

We summarize the equivalent characterizations of ternary positroids in the following theorem.

\begin{thm} \label{thm:pos-tern-minor}
    The following are equivalent for a non-empty positroid $P$.
    \begin{itemize}
        \item[(i)] $P$ is ternary.
        \item[(ii)] $P$ is near-regular.
        \item[(iii)] $P$ is $U^2_5$ and $U^3_5$-free.
        \item[(iv)] $P$ can be obtained by direct sums and $2$-sums from circuits, cocircuits, and positroids isomorphic to $\mathcal{W}^r$ for some $r \geq 2$. 
    \end{itemize}
\end{thm}

\begin{proof}
    We break the proof up by items.

    \begin{itemize}
        \item[$(i) \Leftrightarrow (ii) \Leftrightarrow (iii)$] As shown in~\cite{chidiac2021positroids}, all positroids are gammoids, and as stated in~\cite{oxley2006matroid}, all gammoids are $M(K_4), P_7,$ and $P^*_7$-free. Thus the equivalence of (i)-(iii) immediately follows from Theorem~\ref{thm:tfae:tern-gammoids}.
        
        \item[$(i) \Leftrightarrow (iv)$] Suppose that $P$ is ternary. Then, by Theorem~\ref{thm:pos-direct-sum}, $P$ has a direct sum decomposition into $2$-connected ternary positroids on pairwise disjoint, non-crossing subsets of $E(P)$, $P = \bigoplus^n_{i=1} P_i$. Let $P_i$ be fixed. Then, by Theorem~\ref{thm:pos-canon-tree-iff}, $P_i$ has a canonical tree decomposition into ternary positroids $\{P_{i,j}\}^m_{j=1}$ that are circuits, cocircuits, or $3$-connected. All positroids are grammoids, so for all $j \in [m]$, $P_{i,j}$ is a ternary gammoid. Let $P_{i,j}$ be fixed and suppose that $P_{i,j}$ is $3$-connected and neither a circuit nor a cocircuit. Then, as all $3$-connected matroids on fewer than four elements are circuits or cocircuits, $E(P_{i,j}) \geq 4$, so by Corollary~\ref{cor:tern-3connect}, $P_{i,j} \cong \mathcal{W}^r$ for some $r \geq 2$.

        Suppose instead that $P$ can be obtained by direct sums and $2$-sums from circuits, cocircuits, and positroids that are isomorphic to $\mathcal{W}^r$ for some $r \geq 2$. All circuits and cocircuits are graphic, hence regular and therefore ternary. Furthermore, for all $r \geq 2$, $\mathcal{W}^r$ is ternary. For any field $\mathbb{F}$, direct sums and $2$-sums preserve $\mathbb{F}$-realizability, so $P$ is ternary.
    \end{itemize}
\end{proof}

We now turn our attention to the quaternary positroids. Quaternary matroids have the following forbidden minor characterization.

\begin{thm}[Theorem 1.3 in \cite{geelen2000excluded(2)}] \label{thm:f_4-excluded-minors}
    A matroid $M$ is $\mathbb{F}_4$-realizable if and only if $M$ has no minor isomorphic to any of $U^2_6,U^4_6,P_6,F^-_7, (F^-_7)^*,P_8$, or $P^=_8$.
\end{thm}

We apply the following Corollary~\ref{cor:pos-circ-hyp-relax} and Lemma~\ref{lem:pos-uni} to characterize the quaternary positroids by their minimal set of forbidden minors in Proposition~\ref{prop:pos-quat}.

\begin{cor}[Corollary 3.7 in \cite{bonin2023characterization}] \label{cor:pos-circ-hyp-relax}
    Any circuit-hyperplane relaxation of a connected positroid of rank at least two is a positroid.
\end{cor}

\begin{lemma}[Lemma 11 in \cite{benedetti2022quotients}] \label{lem:pos-uni}
    Let $k,n$ be integers such that $0 \leq k \leq n$. The matroid $U^k_n$ has a positroid ordering.
\end{lemma}

\begin{figure}
    \centering
    \begin{tikzpicture}
        \node at (0,-0.4){$a$};
        \node at (2,-0.4){$b$};
        \node at (4,-0.4){$c$};
        \node at (3.4,1.5){$d$};
        \node at (2,3.4){$e$};
        \node at (0.6,1.5){$f$};
        \node at (2,1.1){$g$};
        
        \draw[fill=black] (0,0) circle (3pt);
        \draw[fill=black] (1,1.5) circle (3pt);
        \draw[fill=black] (2,3) circle (3pt);
        \draw[fill=black] (3,1.5) circle (3pt);
        \draw[fill=black] (4,0) circle (3pt);
        \draw[fill=black] (2,0) circle (3pt);
        \draw[fill=black] (2,1.5) circle (3pt);

        \draw[thick] (0,0) -- (2,3);
        \draw[thick] (0,0) -- (4,0);
        \draw[thick] (2,3) -- (4,0);
        \draw[thick] (1,1.5) -- (3,1.5);
    \end{tikzpicture}
    \caption{A geometric representation of the matroid $P^-_7$.}
    \label{fig:p-7}
\end{figure}

\begin{prop} \label{prop:pos-quat}
    A positroid $P$ is $\mathbb{F}_4$-realizable if and only if $P$ is $U^2_6, U^4_6,$ and $P_6$-free.
\end{prop}

\begin{proof}
    By Theorem~\ref{thm:f_4-excluded-minors}, $P$ is $\mathbb{F}_4$-realizable if and only if $P$ contains no minor isomorphic to any of $U^2_6, U^4_6, P_6, F^-_7, (F^-_7)^*, P_8,$ or $P^=_8$. None of the matroids $F^-_7, (F^-_7)^*$, or $P_8$ are gammoids, hence they do not have positroid orderings. 
    
    Now consider the geometric representation given in Figure~\ref{fig:p-7} of the matroid $P^-_7$, which is a minor of $P^=_8$. The set of circuit-hyperplanes of $P^-_7$ is the set of lines in Figure~\ref{fig:p-7}, $\left\lbrace \{a,f,e\}, \{a,b,c\}, \{c,d,e\}, \{f,g,d\} \right\rbrace$. By way of contradiction, suppose that $P^-_7$ has a positroid ordering. Then, by Theorem~\ref{thm:pos-crossing}, each of the circuit-hyperplanes of $P^-_7$ is a cyclic interval. By Proposition~\ref{prop:pos-dihedral}, up to dihedral action, the positroid ordering on $P^-_7$ is one of the following: $\{a < f < e < \cdots \}, \{a < e < f < \cdots\}, \{a < \cdots < f < e\},$ or $\{a < \cdots < e < f\}$. Suppose that $\{a < f < e < \cdots\}$ or $\{a < \cdots < e < f\}$, then $\{f,g,d\}$ is not a cyclic interval, which is a contradiction. Suppose instead that $\{a < e < f < \cdots\}$ or $\{a < \cdots < f < e\}$, then $\{e,d,c\}$ is not a cyclic interval, which is a contradiction. Thus, $P^-_7$ does not have a positroid ordering. Therefore, by Proposition~\ref{prop:pos-dual-minor}, $P^=_8$ does not have a positroid ordering and $P$ is $P^=_8$-free. It follows that $P$ is $\mathbb{F}_4$-realizable if and only if $P$ is $U^2_6,U^4_6,$ and $P_6$-free. 

    As $P_6$ can be obtained from the rank-$3$ whirl $\mathcal{W}^3$ by circuit-hyperplane relaxations, by Corollary~\ref{cor:pos-circ-hyp-relax}, $P_6$ has a positroid ordering. By Lemma~\ref{lem:pos-uni}, $U^2_6$ and $U^4_6$ both have positroid orderings. Therefore, $\{P_6,U^2_6,U^4_6\}$ is the minimum set of forbidden unordered matroid minors that characterizes the $\mathbb{F}_4$-realizable positroids.
\end{proof}

\subsection{Ternary envelope classes}

In this subsection, we characterize the ternary positroid envelope classes. We begin with the positroid envelope class $\Omega_P$, where $P$ is isomorphic to $U^2_4$, before proceeding to the positroid envelope classes of whirls of rank greater than two.  

\begin{lemma}
    Let $P$ be a positroid isomorphic to $U^2_4$. Then, $\Omega_P$ consists of exactly four matroids: one isomorphic to $M(\mathcal{N}_2$), two isomorphic to $\mathcal{N}^2 \cong M(\mathcal{W}_2)$, and one isomorphic to $\mathcal{W}^2 \cong U^2_4$.
\end{lemma}

\begin{proof}
    Without loss of generality, we may take $P = U^2_4$. Observe that $\mathcal{B}(M(\mathcal{N}_2)) = \left\{\{1,2\},\{1,4\},\{2,3\},\{3,4\}\right\}$ and $\mathcal{C}(M(\mathcal{N}_2)) = \left\{\{1,3\},\{2,4\}\right\}$. Let $M_1$, $M_2$ be the matroids on $[4]$ whose bases are
    \begin{align*}
        \mathcal{B}(M_1) &= \left\{\{1,2\},\{1,3\},\{1,4\},\{2,3\},\{3,4\}\right\}\\
        \mathcal{B}(M_2) &= \left\{\{1,2\},\{1,4\},\{2,3\},\{2,4\},\{3,4\}\right\}.
    \end{align*}
    Then, $\Omega_P = \{P,M(\mathcal{N}_2),M_1,M_2\}$. Furthermore, $M_1$ is obtained from $M(\mathcal{N}_2)$ by the relaxation of the circuit-hyperplane $\{1,3\}$ and $M_2$ is obtained from $M(\mathcal{N}_2)$ by the relaxation of the circuit-hyperplane $\{2,4\}$. It follows that $M_1,M_2 \cong \mathcal{N}^2 \cong M(\mathcal{W}_2)$.
\end{proof}

\begin{lemma} \label{lem:whirl-class-contains-graphic-bases}
     Let $P$ be a positroid isomorphic to $\mathcal{W}^r$, and let $M \in \Omega_P$. Let $R $ be the rim of the edge-ordered wheel graph $\mathcal{W}_r$, such that $P$ can be obtained by the circuit-hyperplane relaxation of $R$ in $M(\mathcal{W}_r)$. Then
   \[\{ B \in \mathcal{B}(P) : \lvert R \setminus B \rvert = 1 \} \subset \mathcal{B}(M).\]
\end{lemma}

\begin{proof}
Without loss of generality, assume that $P$ is a positroid on $[2r]$ and $\pi^:_P = (1,3,\ldots,{2r-1})({2r},{2r-2},\ldots,2)$. Then $P$ is the whirl corresponding to the wheel graph in Figure~\ref{fig:wheel-graph}. For $\mathcal{J}(P) = (J_1,J_2,\ldots,J_{2r})$ and for all $k \in [2r]$, with elements taken modulo ${2r}$, we have
    \begin{equation*}
        J_k =
        \begin{cases}
            \{k\} \cup \{ {k+1}, {k+3},\dots, {k-3}\}, & \text{if } k \text{ is odd}\\
            \{{k+1}\}  \cup \{k, {k+2},\dots,{k-4}\}, & \text{if } k \text{ is even.}
        \end{cases}
    \end{equation*}
Let $M \in \Omega_P$, then $\mathcal{B}(M) \supseteq \{J_{1},J_{2},\ldots,J_{{2r}}\}$. Now, let $1\leq s,t\leq 2r$ with $s$ odd and $t$ even. We will show that the set $X_{s,t}=\{ s \}\cup \{ {t+2}, {t+4},\dots,{t-2} \}$ is also a basis of $M$. If $t\in \{s-1, s-3\}$, then this set is already included in $\mathcal{J}(P)$. We proceed by induction. Suppose that $X_{s,t}$ is a basis of $M$ for any $t \in \{s-1,s-3,\dots, s-a\}$, with $3 \leq a \leq 2r-3$ odd. Consider the two bases $X_{s,s-a}$ and $X_{s-2,s-a-2}$. Then $X_{s-2,s-a-2} \setminus X_{s,s-a} = \{ s-2,s-a \}$ and $X_{s,s-a} \setminus X_{s-2,s-a-2} = \{ s,s-a-2 \}$. Therefore, by the basis exchange property, either $(X_{s,s-a} \setminus \{ s-a-2 \}) \cup \{s-2\}$ or $(X_{s,s-a}\setminus \{s-a-2\} )\cup \{s-a\} = X_{s,s-a-2}$ is a basis of $M$. Since the circuit $\{s,s-1,s-2  \} $ is a subset of $(X_{s,s-a}\setminus \{s-a-2\} )\cup \{s-2\}$, we must have that $X_{s,s-a-2}$ is a basis of $M$ for any odd $1\leq s \leq 2r-1$, which completes the proof by induction.
\end{proof}

There is a matroid $N$ in $ \Omega_P$ with bases
\begin{align*}
    B(N) &= \{ B \in \mathcal{B}(P) : \lvert R \setminus B \rvert {=} 1 \}\\
    &= \{  X_{s,t} : 1\leq s,t \leq 2r, \; s \mbox{ odd}, t \mbox{ even}\}
\end{align*}
as in Lemma~\ref{lem:whirl-class-contains-graphic-bases}. This matroid $N$ is isomorphic to the graphic matroid $M(\mathcal{N}_r)$ from Figure~\ref{fig:wheel-graph}. There is also a matroid isomorphic to $\mathcal{N}^r$ with $\mathcal{B}(\mathcal{N}^r)=\mathcal{B}(M(\mathcal{N}_r)) \cup \{ R \}$, obtained by the circuit-hyperplane relaxation of $R$ in $M(\mathcal{N}_r)$. We will now show that the only other matroids, besides the one isomorphic to $M(\mathcal{N}_r)$ and the one isomorphic to $\mathcal{N}^r$, in the envelope class of a positroid $P \cong \mathcal{W}^r$ for some $r \geq 2$ are $P$ and a matroid isomorphic to $M(\mathcal{W}_r)$.

\begin{lemma} \label{lem:whirl-class-k-spokes}
    Let $P$ be a positroid isomorphic to $\mathcal{W}^r$ for some $r \geq 2$, and let $M \in \Omega_P$. Suppose that there exists $B \in \mathcal{B}(M)$ such that $\lvert R \setminus B\rvert >1$, where $R $ is the rim of the edge-ordered wheel graph $\mathcal{W}_r$ such that $P$ can be obtained by the circuit-hyperplane relaxation of $R$ in $M(\mathcal{W}_r)$. Then
    \begin{equation*}
        \mathcal{B}(M(\mathcal{W}_r)) \subseteq \mathcal{B}(M).
    \end{equation*}
    Therefore, $M$ is isomorphic to $M(\mathcal{W}_r)$ or $\mathcal{W}^r$.
\end{lemma}

\begin{proof}
Without loss of generality, assume $P$ is a positroid on $[2r]$ and $\pi^:_P = (1,3,\ldots,2r-1)(2r,2r-2,\ldots,2)$. We will show that for any $1\leq k \leq r$, we have
\begin{equation}\label{eq:basesk}
        \{ B \in \mathcal{B}(P) : \lvert R \setminus B \rvert = k \} \subset \mathcal{B}(M).
    \end{equation}
We let $X_{s_1,\dots,s_k,t_1,\dots,t_k}=\{s_1,\dots,s_k\} \cup (\{2,4,\dots,2r\}\setminus \{t_1,\dots,t_k\})$, for $1\leq s_1,\dots,s_k,t_1,\dots,t_k \leq 2r $, and $s_1,\dots,s_k$ odd and $t_1,\dots,t_k $ even. First we note that for such a set to be a basis of $P$, there cannot be a pair $s_i,s_j$ such that $s_i<_{s_i} s_j <_{s_i} t_1,\dots,t_k$, or else the set contains a cycle in $\mathcal{W}_r$ and therefore a circuit in $\mathcal{W}^r$. We claim that any such set $X_{s_1,\dots,s_k,t_1,\dots,t_k}$ where the odd and even numbers interlace is indeed a basis. 

We proceed by induction. Our base case, $k=1$, follows from Lemma~\ref{lem:whirl-class-contains-graphic-bases}. Let $k>1$ and suppose the claim~(\ref{eq:basesk}) holds for $1,\dots,k-1$. We start by showing that $M$ has some basis $B$ such that $|R \setminus B|=k$. If $k=2$, we note that, by assumption, $M$ has a basis $B$ with $|R \setminus B|>1$. If $|R \setminus B|>2$, we have some $\{s,s'\} \subseteq B \setminus R$. Consider the basis $X_{s,t}$ for any $t$, which exists by Lemma~\ref{lem:whirl-class-contains-graphic-bases}. By the basis exchange property, we can add an element from $X_{s,t}$ to $B \setminus \{s'\}$ to obtain a new basis $B'$ such that $|R \setminus B'| < |R \setminus B|$. Repeating this step eventually yields a basis $\tilde{B}$ such that $|R \setminus \tilde{B}|=2$. Now, suppose that $k>2$. Let $X=X_{s_1,\dots,s_{k-1},t_1,\dots,t_{k-1}}$ be a basis of $M$. As long as $k<r$, there must be an $x\in \{s_1,\dots,s_{k-1},t_1,\dots,t_{k-1}\}$ such that $x+1,x+2 \notin \{s_1,\dots,s_{k-1},t_1,\dots,t_{k-1}\}$. By symmetry and duality, without loss of generality we may assume that $x=s_1$. Then, $X'=X_{s_1+2,s_2,\dots,s_{k-1},s_1+1,t_1,\dots,t_{k-2}}$ is a basis of $M$ by the inductive hypothesis. Then, either $(X\setminus \{ s_1+1\})\cup \{t_{k-1}\}$ or $(X \setminus \{ s_1+1\})\cup \{  s_1+2\}=X_{s_1,s_1+2,s_2,\dots,s_{k-1},s_1+1,t_1,\dots,t_{k-1}}$ is a basis of $M$. Since the former contains the circuit $\{s_{k-1}, s_{k-1} + 1, s_{k-1} + 3,\ldots,s_1 - 1,s_1\}$, we must have the latter. Therefore, for $1\leq k \leq r$, there exists a basis $B$ in $M$ with $|R \setminus B|=k$.

Now, we will show that, for $1\leq k \leq r$, if $M$ has a basis $B$ with $|R \setminus B|=k$, then it has every possible such basis. Suppose that $X=X_{s_1,\dots,s_k,t_1,\dots,t_k} \in \mathcal{B}(M)$ and, without loss of generality, that we have 
\[ s_1 <_{s_1} t_1 <_{s_1} \dots <_{s_1} s_k <_{s_1} t_k .\]
We claim that if for any $x\in \{s_1,\dots,s_k,t_1,\dots,t_k\}$, $x+2$ is in the same position as $x$ in this order, then we can replace $x$ with $x+2$ to obtain a new basis. This move is sufficient to then obtain any pair of such interlacing sets. By symmetry and duality, it suffices to show this for $s_1$. Suppose that $t_k <_{s_1} s_1, s_1+2 <_{s_1} t_1$. By the inductive hypothesis, $X'=X_{s_1+2,s_2,\dots,s_{k-1},t_1,\dots,t_{k-1}} \in \mathcal{B}(M)$. Thus, either $(X' \setminus \{t_k\})\cup \{s_1\}$ or $(X' \setminus \{t_k\}) \cup \{s_k\} = X_{s_1+2,s_2,\dots,s_k,t_1,\dots,t_k} $ is in $\mathcal{B}(M)$. Since the circuit $\{s_1, s_1+1,s_1+2\}$ is contained in the former, we must have the latter, which completes the proof.
\end{proof}

\begin{cor}\label{cor:pos-whirl-class}
Let $P$ be a positroid isomorphic to $\mathcal{W}^r$ for some $r \geq 2$. Then matroid isomorphism yields the bijection 
\[\Omega_P \cong \{ M(\mathcal{N}_r),\mathcal{N}^r ,M(\mathcal{W}_r),\mathcal{W}^r   \}  .\]
\end{cor}

In Thereom~\ref{thm:pos-tern-envelope}, we extend the characterizations of the ternary positroids from Theorem~\ref{thm:pos-tern-minor}.

\begin{thm} \label{thm:pos-tern-envelope}
    The following are equivalent for a non-empty positroid $P$.
    \begin{itemize}
        \item[(i)] $P$ is ternary.
        \item[(ii)] $\Omega_P$ has a direct sum and canonical tree decomposition into $\{\Omega_{P_i}\}^n_{i=1}$ such that for all $i \in [n]$, $|\Omega_{P_i}| = 1$ or for some $r \geq 2$ there is a bijection given by matroid isomorphisms from $\Omega_{P_i}$ to the set $\{M(\mathcal{N}_r),\mathcal{N}^r,M(\mathcal{W}_r),\mathcal{W}^r\}$.
        \item[(iii)] For every field $\mathbb{F}$, $\Pi^{\circ}_{\mathbb{F}}(P)$ is the disjoint union of near-regular (ternary) matroid strata.
        \item[(iv)] $\Pi_{\mathbb{R}}(P)$ is the disjoint union of near-regular (ternary) matroid strata.
    \end{itemize}
\end{thm}

\begin{proof}
    We break the proof up by items.

    \begin{itemize}
        \item[$(i) \Leftrightarrow (ii)$] Suppose that $P$ is ternary. Then, by Theorem~\ref{thm:pos-direct-sum}, $P$ has a direct sum decomposition into $2$-connected ternary positroids on pairwise disjoint, non-crossing subsets of $E(P)$, $P = \bigoplus^n_{i=1} P_i$. This induces, by Proposition~\ref{prop:direct-sum-class-decomp}, a direct sum decomposition of the positroid envelope class associated to $P$, $\Omega_P = \bigoplus^n_{i=1} \Omega_{P_i}$. Let $P_i$ be fixed. Then, by Theorem~\ref{thm:pos-canon-tree-iff}, $P_i$ has a canonical tree decomposition into ternary positroids $\{P_{i,j}\}^m_{j=1}$ that are circuits, cocircuits, or $3$-connected. This induces, by Proposition~\ref{prop:2sum-class-decomp}, a canonical tree decomposition of $\Omega_{P_i}$ into positroid envelope classes $\{\Omega_{P_{i,j}}\}^m_{j=1}$. Now let $P_{i,j}$ be fixed. By Theorem~\ref{thm:pos-graphic}, if $P_{i,j}$ is binary, then $|\Omega_{P_{i,j}}| = 1$. If $P_{i,j}$ is non-binary, then $E(P_{i,j}) \geq 4$. By Corollary~\ref{cor:tern-3connect}, $P_{i,j}$ is isomorphic to $\mathcal{W}^r$ for some $r \geq 2$. Then, by Corollary~\ref{cor:pos-whirl-class}, there is a bijection given by matroid isomorphism from $\Omega_{P_{i,j}}$ to the set $\{M(\mathcal{N}_r),\mathcal{N}^r,M(\mathcal{W}_r),\mathcal{W}^r\}$.

        Suppose instead that $\Omega_P$ has a direct sum decomposition into $\{\Omega_{P_i}\}^n_{i=1}$ each with a canonical tree decomposition into $\{\Omega_{P_{i,j}}\}^{m_i}_{j=1}$ such that for all $i \in [n], j \in [m_i]$, $|\Omega_{P_{i,j}}| = 1$ or for some $r \geq 2$ there is a bijection given by matroid isomorphisms from $\Omega_{P_{i,j}}$ to the set $\{M(\mathcal{N}_r),\mathcal{N}^r,M(\mathcal{W}_r),\mathcal{W}^r\}$. Let $\Omega_{P_{i,j}}$ be fixed. If $|\Omega_{P_{i,j}}| = 1$, then by Theorem~\ref{thm:pos-graphic}, $P_{i,j}$ is regular and therefore ternary. If $\Omega_{P_{i,j}}$ is in bijection with $\{M(\mathcal{N}_r),\mathcal{N}^r,$ $M(\mathcal{W}_r),$ $\mathcal{W}^r\}$, then $P_{i,j}$ is isomorphic to $\mathcal{W}^r$ for some $r \geq 2$, hence $P_{i,j}$ is ternary. Then, $P$ is obtained from direct sums and $2$-sums of ternary matroids and is therefore ternary.

        \item[$(i) \Leftrightarrow (iii)$] Suppose that $P$ is ternary and therefore $\Omega_P$ has a direct sum decomposition into positroid envelope classes $\{\Omega_{P_i}\}^n_{i=1}$ each with a $2$-sum decomposition into envelope classes $\{\Omega_{P_{i,j}}\}^{m_i}_{j=1}$ such that for each $\Omega_{P_{i,j}}$, $|\Omega_{P_{i,j}}| = 1$ or for some $r \geq 2$ there is a bijection given by matroid isomorphisms from $\Omega_{P_{i,j}}$ to the set $\{M(\mathcal{N}_r),\mathcal{N}^r,$ $M(\mathcal{W}_r),$ $\mathcal{W}^r\}$. Let $\Omega_{P_{i,j}}$ be fixed. If $|\Omega_{P_{i,j}}| = 1$, then by Theorem~\ref{thm:pos-graphic}, $\Omega_{P_{i,j}}$ consists of a single regular and therefore near-regular matroid. The set $\{M(\mathcal{N}_r),\mathcal{N}^r,$ $M(\mathcal{W}_r),$ $\mathcal{W}^r\}$ consists entirely of near-regular matroids, so if $|\Omega_{P_{i,j}}| \neq 1$, then $\Omega_{P_{i,j}}$ consists entirely of near-regular matroids. Thus, $\Omega_P$ consists entirely of matroids that are direct sums and $2$-sums of near-regular matroids, hence are near-regular and therefore ternary. Let $\mathbb{F}$ be a field, then $\Pi^{\circ}_{\mathbb{F}}(P) = \bigsqcup_{M \in \Omega_P} S_M(\mathbb{F})$. Thus, $\Pi^{\circ}_{\mathbb{F}}(P)$ is the disjoint union of near-regular and therefore ternary matroid strata.

        Suppose instead that for every field $\mathbb{F}$, $\Pi^{\circ}_{\mathbb{F}}(P)$ is the disjoint union of near-regular and therefore ternary matroid strata. Then in particular, for the field $\mathbb{R}$, $S_P(\mathbb{R}) \subseteq \Pi^{\circ}_{\mathbb{R}}(P)$ is a ternary matroid stratum, so $P$ is ternary.

        \item[$(i) \Leftrightarrow (iv)$] Suppose that $P$ is ternary of rank-$r$. Without loss of generality, take $P$ to be a positroid on $[n]$. Let $M$ be an $\mathbb{R}$-linear matroid, with positroid envelope $Q$, such that $\mathbb{1} : P \to M$ is a rank-preserving weak-map. Let $\mathcal{J}(M) = (J_1,J_2,\ldots,J_n) = \mathcal{J}(Q)$ and let $\mathcal{J}(P) = (J'_1,J'_2,\ldots,J'_n)$. Then, for all $i \in [n]$, $J'_i \leq J_i$. By Theorem~\ref{thm:pos-intersect}, we obtain
        \begin{align*}
            \mathcal{B}(Q) &= \bigcap^n_{i=1} \left\{ B \in {[n] \choose r} : J_i \leq_i B \right\}\\
            &\subseteq \bigcap^n_{i=1} \left\{ B \in {[n] \choose r} : J'_i \leq_i B \right\} = \mathcal{B}(P).
        \end{align*}
        Therefore, $\mathbb{1} : P \to Q$ is a rank-preserving weak map. As $P$ is $U^2_5$ and $U^3_5$-free, then by Corollary~\ref{cor:uni-free}, $Q$ is $U^2_5$ and $U^3_5$-free, hence $Q$ is near-regular and $M \in \Omega_Q$ is near-regular. It follows that $\Pi_{\mathbb{R}}(P)$ is the disjoint union of near-regular and therefore ternary matroid strata.

        Suppose instead that $\Pi_{\mathbb{R}}(P)$ is the disjoint union of near-regular and therefore ternary matroid strata. Then, $S_P(\mathbb{R}) \subseteq \Pi_{\mathbb{R}}(P)$ is a ternary matroid stratum, hence $P$ is ternary.
    \end{itemize}
\end{proof}

By Theorem~\ref{thm:pos-tern-envelope} and Proposition~\ref{prop:semilattices-direct-sums-2sums}, every ternary positroid envelope class is a lattice.

\begin{thm} \label{thm:tern-envelope-lattice}
    For a ternary positroid $P$, the positroid envelope class $\Omega_P$ is a lattice under the rank-preserving weak map partial order.
\end{thm}

We provide the following formula for enumerating the matroids in a ternary positroid envelope class using direct sum and $2$-sum decompositions of envelope classes as given by Propositions~\ref{prop:direct-sum-class-decomp} and~\ref{prop:2sum-class-decomp}.

\begin{cor}
    Let $P$ be a ternary positroid and $w$ be equal to the number of non-binary $3$-connected components of $P$. Then, $|\Omega_P| = 4^w$.
\end{cor}

\begin{proof}
By Theorem~\ref{thm:pos-tern-minor}, a ternary positroid can be decomposed via direct sums and 2-sums into circuits, cocircuits, and positroids isomorphic to $\mathcal{W}^r$ for some $r \geq 2$. By Propositions~\ref{prop:direct-sum-class-decomp} and ~\ref{prop:2sum-class-decomp}, this corresponds directly to a direct sum and $2$-sum decomposition of the envelope class. By Theorem~\ref{thm:pos-graphic} and Corollary~\ref{cor:pos-whirl-class}, the envelope classes of the circuits and cocircuits have cardinality $1$, and those of the whirls have cardinality $4$. Therefore, there are $4^w$ matroids in the envelope class, where $w$ is the number of non-binary $3$-connected components of $P$.
\end{proof}

\section*{Acknowledgments}
We thank Allen Knutson, Puck Rombach and Melissa Sherman-Bennett for helpful conversations.

\end{document}